\documentclass[12pt]{amsart}
\usepackage[utf8]{inputenc}
\usepackage{a4wide}
\usepackage{lmodern}

\usepackage{hyperref}
\usepackage{dsfont}
\usepackage{stmaryrd}
\usepackage{todonotes}
\usepackage{cite}
\usepackage[capitalize]{cleveref}

\usepackage{amsmath,amsfonts,amssymb,amsthm}
\usepackage{bbm}
\usepackage{pgf,tikz,pgfplots}
\pgfplotsset{compat=1.14}
\usepackage{mathrsfs}
\usetikzlibrary{arrows}

\allowdisplaybreaks

\usepackage{amsmath}
\usepackage{amssymb}
\usepackage{amsthm}
\usepackage{amsfonts}

\newcommand{\N}{\mathbb{N}}

\newcommand{\C}{\mathbb{C}}

\theoremstyle{plain}
\newtheorem{thm}{Theorem}[section]
\newtheorem{lem}[thm]{Lemma}
\newtheorem{prop}[thm]{Proposition}
\newtheorem{cor}[thm]{Corollary}
\newtheorem{dfn}[thm]{Definition}
\newtheorem*{conv}{Convention}
\newtheorem*{cond}{\bf{Condition (C)}}

\theoremstyle{definition}

\theoremstyle{remark}
\newtheorem*{rem}{Remark}

\newtheorem*{nota}{Notation}

\usepackage{xspace}
\def\Vs{\mathfrak{N}\xspace}
\def\Vsr{$\mathfrak{N}$\xspace}

\title[Random graphs with a prescribed modular decomposition]{Subgraph densities and scaling limits of random graphs with a prescribed modular decomposition}
 \author{Théo Lenoir}

\begin{document}
\maketitle

\begin{abstract}
We consider large uniform labeled random graphs in different classes
with prescribed decorations in their modular decomposition. Our main result is the estimation of the number of copies of every graph as an induced subgraph. As a consequence, we obtain the convergence of a uniform random graph in such classes to a Brownian limit object in the space of graphons.

Our proofs rely on combinatorial arguments, computing generating series using the symbolic method and deriving asymptotics using singularity
analysis.

\end{abstract}

\section{Introduction}
\subsection{Motivation}

This article lays at the interface of graph theory and combinatorial probability.

 The first motivation comes from modular decomposition, which is a standard tool in graph theory. This notion will be reviewed in a comprehensive way in \cref{sec2}. Roughly speaking, it is a decomposition of a graph into disjoint subsets of vertices called \emph{modules}. The modules can be thought of as generalizations of connected components. Since a module can contain a proper module, the modular decomposition is a recursive decomposition of a graph. It is analogous to prime factor decomposition of integers: graphs without trivial modules cannot be decomposed, and thus will be called prime for the modular decomposition.

From an algorithmic perspective, the modular decomposition can be computed in linear time \cite{linear}, making it an essential tool for quickly solving various problems such as determining belonging to several graph classes (\emph{e.g.} cographs, $P_4$-sparse graphs, permutations graphs, etc). We refer the readers to \cite{graphclasses} and \cite{modular} for more details.

The study of modular decomposition from a probabilistic point of view started very recently. The case with no prime graphs, \emph{i.e.} cographs, has been studied simultaneously in \cite{bassino2021random} and \cite{bs}. In a somewhat different context, the author \cite{lenoir} used modular decomposition to study the case of graphs with few $P_4$'s (for $k\geq 2$, $P_k$ denotes the graph consisting of a line with $k$ vertices). The aim of this article is to generalize these results: we will study the structure of large typical random graphs defined by a set of prime graphs allowed in its modular decomposition.

The second motivation is the study of classes of graphs with forbidden induced subgraphs. Indeed, our results can be interpreted in terms of forbidden induced subgraphs, as the interdiction of some prime induced subgraphs is heavily connected to the modular decomposition. Examples of classes covered by our results include those with a finite number of prime induced subgraphs allowed, or classes where all prime induced subgraphs are lines.

The asymptotic study of random graphs has been a well-established subject of research for several decades. Depending on the graph models and graph parameters of interest, several notions of limit can be considered.
\begin{itemize}
\item A first point of view consists in studying the local limit of the graph around a typical vertex. This gives information on local neighborhoods (\emph{e.g.} degrees). See \cite{remco} for more details.
\item Alternatively, we can adopt a global viewpoint, in at least two very different frameworks.
\begin{itemize}
\item We may consider the scaling limit of graphs viewed as metric spaces, by assigning distances to edges. This gives estimates of typical distances or diameters. In this context, a universal object arises: the Continuum Random Tree (CRT). For examples, refer to \cite{aldous2,goldschmidt,panagiotou}.
\item Instead, we may consider occurrences of induced subgraphs, the proper framework for which is graphons. Introduced in \cite{graphons}, graphon convergence can be seen as convergence of renormalized adjacency matrices for the so-called cut metric (a good reference on graphon theory is \cite{lovasz}). This sometimes also gives estimation on typical degrees \cite{bassino2021random}, on extremal statistics (cliques and independent sets as in \cite{indset}), or on the growth rate of classes of graphs \cite{hatami}.

\end{itemize}
\end{itemize}
The last two frameworks answer very different questions and are adapted respectively to sparse graphs (graphs with a linear number of edges) and dense graphs (graphs with a quadratic number of edges).

\subsection{Main results}

For a set of prime graphs $\mathcal{P}$, let $\mathcal{G}_{\mathcal{P}}$ be the set of graphs $G$ such that every prime subgraph in the modular decomposition of $G$ is in $\mathcal{P}$: the aim of this article is to study the structure of a uniform graph of size $n$ in $\mathcal{G}_{\mathcal{P}}$ when $n$ tends towards $ +\infty$. We only consider the case where $\mathcal{P}$ is stable by automorphism. Let $P(z)=\sum\limits_{H\in\mathcal{P}}\frac{z^{|H|}}{|H|!}$ be the exponential generating function associated to a given set $\mathcal{P}$ and let $R_0\in[0,+\infty]$ be the radius of convergence of $P$. Set $\Lambda(w):=P(\exp(w)-1)+\exp(w)-1-w$, $\Lambda$ is a power series with nonnegative coefficients. An important condition which will be crucial on $\Lambda$ is the following:

\begin{cond}\label{condi}\hspace{2cm}$R_0>0$\ and\ 
$\Lambda'\left(\log(1+R_0)\right)>1$ 
\end{cond}
\noindent (By convention, we set $\Lambda'\left(\log(1+R_0)\right)=+\infty$ if the series $\Lambda'$ diverges at $\log(1+R_0)$.) 
This condition implies that the number of graphs of size $n$ in $\mathcal{P}$ does not grow too fast, see \cref{conditionc} for more details.

Let $\mathbf{G}^{(n)}$ be a graph of size $n$ taken uniformly at random in $\mathcal{G}_{\mathcal{P}}$. For any graph $H$, denote by $\mathrm{Occ}_{H}(\mathbf{G}^{(n)})$ the number of labeled induced subgraphs of $\mathbf{G}^{(n)}$ isomorphic to $H$. Our first contribution is the estimation of the expected number of copies of $H$ in $\mathbf{G}^{(n)}$, \emph{i.e.} $\mathbb{E}[\mathrm{Occ}_{H}(\mathbf{G}^{(n)})]$. 

\begin{thm}\label{thmintro2}
Let $\mathcal{P}$ be a set of prime graphs stable by automorphism such that Condition (C) holds. Let $\mathbf{G}^{(n)}$ be a graph of size $n$ taken uniformly at random in $\mathcal{G}_{\mathcal{P}}$.

Then for every graph $H$ there exists $K_H\geq0$ such that:
$$\mathbb{E}[\mathrm{Occ}_{H}(\mathbf{G}^{(n)})]\sim K_H n^{|H|-\beta(H)}$$
where the constant $K_H$ is given in \Cref{thmprime} and $\beta(H)$ is defined in \cref{betaG}.
\end{thm}

Let us briefly comment on the theorem.
\begin{itemize}
\item The quantity $\beta(H)$ is nonnegative, does not depend on $\mathcal{P}$, and is equal to $0$ if and only if $H$ is a cograph. Thus at fixed size, the graphs $H$ that appear the most as induced subgraphs of $\mathbf{G}^{(n)}$ are the cographs. Informally, the quantity $\beta(H)$ can be thought of how far $H$ is from cographs.

\item We have that $K_H>0$ if and only if each prime subgraph in the modular decomposition of $H$ is an induced subgraph of an element of $\mathcal{P}$. This criterion will follow from the expression of $K_H$ given in \Cref{thmprime}.

\item Applying \cref{thmintro2} with the graph $H$ consisting of a single edge, we get that $\mathbb{E}[\#\text{edges}\ \text{of}\ \mathbf{G}^{(n)}]= \Theta(n^2)$: thus $\mathbf{G}^{(n)}$ is in the dense regime.

\end{itemize}

\Cref{thmintro2} has an important consequence in terms of scaling limit: since we are in the dense regime, we use the framework of graphons. Graphon convergence is equivalent to the joint convergence of subgraphs density. More formally, Diaconis and Janson provided in \cite{janson} a criterion for the convergence of random graphs in the sense of graphons: the convergence of a family $(\mathbf{H}^{(n)})_{n\geq 1}$ of random graphs is characterized by the convergence in distribution of $\frac{\mathrm{Occ}_{H}(\mathbf{H}^{(n)})}{n^{|H|}}$ for every finite graph $H$. Therefore \Cref{thmintro2} ensures that $\mathbf{G}^{(n)}$ has a limit in the sense of graphons. All the needed materials for the convergence of graphons will be presented in \cref{graphon}.

For the case of cographs (\emph{i.e.} $\mathcal{P}=\varnothing$), studied simultaneously in \cite{bassino2021random} and \cite{bs}, the authors exhibit a Brownian limit object for a uniform cograph, called the Brownian cographon, which can be explicitly constructed from the Brownian excursion.

Our second contribution is that, for all sets $\mathcal{P}$ satisfying Condition (C), the scaling limit is a one parameter deformation of the Brownian cographon. This answers partially a question in \cite{bassino2021random}.

\begin{thm}\label{cgbintro}
Let $\mathbf{G}^{(n)}$ be a graph of size $n$ taken uniformly at random in $\mathcal{G}_{\mathcal{P}}$. Under Condition (C), $\mathbf{G}^{(n)}$ converges in distribution towards $\mathbf{W}^{p}$ in the sense of graphons, where $\mathbf{W}^{p}$ is the Brownian cographon of parameter $p$, and $p$ is a parameter in $[0,1]$ depending on $\mathcal{P}$ whose expression is given in \cref{parameterp}.

\end{thm}

This theorem shows $\mathbf{W}^{p}$ for various $p\neq 1/2$ as a scaling limit of natural combinatorial objects.

\subsection{Discussion on Condition (C)}\label{conditionc}

First observe that Condition (C) is verified in many situations, for example when:

\begin{itemize}
    \item $\mathcal{P}$ is finite, \emph{i.e.} $P$ is a polynomial;
    \item $P$ has a radius of convergence greater than $1$;
    \item  $P'$ diverges in $R_0$;
    \item the number of graphs in $\mathcal{P}_n$ grows as $Bn!R^nn^{a}$ with $B,R>0$ and $a\geq -2$;
    \item $P$ is rational.
\end{itemize}

As we will see in \cref{cor1}, Condition (C) implies that the graphs in $\mathcal{G}_{\mathcal{P}}$ all share the universal exponent $-3/2$ in the sense that the number of graphs of size $n$ in $\mathcal{G}_{\mathcal{P}}$ is asymptotically equivalent to 
$$D \frac{n!}{R^nn^{3/2}},$$
with $D>0$ and $R>0$ explicit constants depending on $\mathcal{P}$, as proved in \cref{cor1}. All the details on the exponential generating functions and asymptotics will be given in \cref{sec4}.

Let us illustrate the connection between \cref{cgbintro} and our initial motivation for graphs with forbidden induced subgraphs.
\begin{itemize}
    \item The case where all but finitely many prime induced subgraphs are forbidden corresponds to a family of graphs $\mathcal{G}_{\mathcal{P}}$ where $\mathcal{P}$ is finite thus Condition (C) is clearly verified.
    \item The case where all prime subgraphs except lines are forbidden corresponds to the family of graphs $\mathcal{G}_{\mathcal{P}}$ where $\mathcal{P}$ is the set of lines of size at least $4$. Thus $P=\frac{z^4}{2(1-z)}$ and Condition (C) is verified. The parameter $p$ of \cref{cgbintro} is approximately $0.288$.
\end{itemize}

This work can be seen as the counterpart for graphs to the study of limits of substitution-closed permutations in \cite{permuton}. However multiple additional difficulties (graph automorphisms, encoding of graphs in $\mathcal{G}_{\mathcal{P}}$ by \emph{non}-plane trees through the modular decomposition) arise, making the enumeration of $\mathcal{G}_{\mathcal{P}}$ more involved. 

Condition (C) is necessary in \cref{cgbintro} in the sense that if it is not verified, other behaviors can appear: under other general assumptions on the generating series, we can get a very different scaling limit where each induced subgraph may appear in the limiting object. In particular, $\mathbb{G}_n$ does not converge to the Brownian cographon.

\subsection{Proof strategy}

Proofs are essentially combinatorial: we use the tree encoding of the modular decomposition to obtain exact enumerations for a large family of graph classes. We exploit those enumerative results with tools from analytic combinatorics to get asymptotic estimates like the ones of the number of graphs of size $n$ in each class.

The most challenging part of the proofs is the one of \cref{thmcoeur}, giving the combinatorial decomposition of graphs in $\mathcal{G}_{\mathcal{P}}$ with a given induced subgraph. This theorem is the key theorem to get the needed asymptotic estimates.

\subsection{Outline of the paper}

\begin{itemize}
    \item In \Cref{sec2} we define the encoding of graphs with trees, the modular decomposition, which is used throughout the different proofs.
    \item  \Cref{graphon} presents the necessary material on graphons. 
    \item \Cref{sec4,sec5} are about calculating generating series related to our graph classes: in \Cref{sec4} we compute several generating series, and obtain the asymptotic of the number of graphs of size $n$ under Condition (C) and \Cref{sec5} deals with the generating series of graphs with a given induced subgraph.
    \item \Cref{lastsection} concludes the proofs of \Cref{cgbintro} and \Cref{thmintro2}.
\end{itemize}

\section{Background on modular decomposition}\label{sec2}
This section is composed of classical results around modular decomposition. In the last subsection, we prove an enumerative result that will be useful at the end of the paper. 
\subsection{Labeled graphs}
In the following all the graphs considered are simple and finite.  Each time a graph $G$ is defined, we denote by $V$ its set of vertices and $E$ its set of edges. Whenever there is an ambiguity, we denote by $V_G$ (resp.~$E_G$) the set of vertices (resp.~edges) of $G$.

\begin{dfn}

We say that $G=(V,E)$ is a \emph{weakly-labeled graph} if every element of $V$ has a distinct label in $\N$ and that $G=(V,E)$ is a \emph{labeled graph} if every element of $V$ has a distinct label in $\{1,\dots, |V|\}$. 

The \emph{size of a graph} $G$, denoted by $|G|$, is its number of vertices. 

The \emph{minimum of a graph} $G$, denoted $\mathrm{min}(G)$, is the minimal label of its vertices.

\end{dfn}

In the following, every graph is labeled, otherwise we mention explicitly that the graph is weakly-labeled.

\begin{rem}
We do not identify a vertex with its label. A vertex of label $i$ is denoted $v_i$. The label of a vertex $v$ is denoted $\ell(v)$.
\end{rem}

\begin{dfn}\label{red}
For any weakly-labeled object (graph or tree) of size $n$, we call \emph{reduction} the operation that reduces its labels to the set $\{1,\dots,n\}$ while preserving the relative order of the labels.
\end{dfn}

For example if $G$ labels $2,4,12,63$ then the reduced version of $G$ is a copy of $G$ in which  $2, 4, 12, 63$ are respectively replaced by $1,2,3,4$.

\begin{dfn}
Let $G$ be a graph and $\pi$ be a permutation of $\{1,\dots,|G|\}$. The \emph{$\pi$-relabeling} of $G$ is the graph $G'$ such that: 
\begin{itemize}
    \item $V_{G'}=V_G$
    \item for every vertex $v$ in $V_{G'}$, we replace the label of the leaf $v$ by $\pi(\ell(v))$. 
    
\end{itemize}

We write $G\sim G'$ if $|G|=|G'|$ and there exists a permutation $\pi$ of $\{1,\dots, |G|\}$ such that $G$ is isomorphic to the $\pi$-relabeling of $G'$.

\end{dfn}
Note that $\sim$ is an equivalence relation.

\begin{dfn}[Induced subgraph]\label{d12}
Let $G$ be a graph, $k$ a positive integer and $\mathfrak{I}$ a partial injection from the set of labels of $G$ to $\N$. The \emph{labeled subgraph $G_\mathfrak{I}$} of $G$ induced by $\mathfrak{I}$ is defined as:
\begin{itemize}
    \item The vertices of $G_\mathfrak{I}$ are the vertices of $G$ whose label $\ell$ is in the domain of $\mathfrak{I}$. For every such vertex, we replace the label $\ell$ of the vertex by $\mathfrak{I}(\ell)$;
    \item For two vertices $v$ and $v'$ of $G_\mathfrak{I}$, $(v,v')$ is an edge of $G_\mathfrak{I}$ if and only if it is an edge of $G$.
\end{itemize}
\end{dfn}

\begin{dfn}
For every pair of graphs $(G,H)$, let \emph{$\mathrm{Occ}_{G}(H)$} be the number of partial injection $\mathfrak{I}$ from the vertex labels of $H$ to $\N$ such that $H_\mathfrak{I}$ is isomorphic to $G$.

\end{dfn}

\begin{dfn}
Let $k$ be a nonnegative integer We say that $G$ is a \emph{graph with $k$ blossom} if, for every $j\in\{1,\dots, k\}$,  exactly one vertex of $G$ is labeled $*_j$, and the remaining vertices have a distinct label in $\{1,\dots,|V|-k\}$. 

\end{dfn}

\begin{rem}
For $k=0$ a graph with $0$ blossom is simply a graph.
\end{rem}

\begin{dfn}\label{defblo}

Let $G$ be a graph with $k$ blossoms and $v$ a vertex of $G$ which is not a blossom. We define \emph{$\mathrm{blo}_{v}(G)$} to be the labeled graph obtained after the following transformations:
\begin{itemize}

    \item $v$ is now labeled $*_{k+1}$;
    \item the graph obtained is replaced by its reduction as defined in \Cref{red}.

\end{itemize}
\end{dfn}

\subsection{Encoding graphs with trees}

A key construction to make the modular decomposition effective is the graph substitution (also called substitution-composition).

\begin{dfn}[Graph substitution]\label{2.3}
Let $G$ be a graph of size $n$ and $H_1,\dots, H_{n}$ be weakly-labeled graphs such the vertices of $H_1,\dots,H_n$ have mutually distinct labels. The graph \emph{$G[H_1,\dots, H_{n}]=(V,E)$} is the graph whose set of vertices is $V=\bigcup_{i=1}^n V_{H_i}$ and such that:
\begin{itemize}
\item for every $i\in \{1,\dots,n\}$ and every pair $(v,v')\in V_{H_i}^2$, $\{v,v'\}\in E$ if and only if $\{v,v'\}\in E_{H_i}$;
\item For every $(i,j)\in \{1,\dots,n\}$ with $i\neq j$, and every pair $(v,v')\in V_{H_i}\times V_{H_j}$, $\{v,v'\}\in E$ if and only if $\{v_i,v_j\}\in E_G$.
\end{itemize}
\end{dfn}

\begin{nota}

In the following we use the shortcut $\oplus$ for the complete graph of size $n$. Thus $\oplus[H_1,\dots ,H_n]$ is the graph obtained from copies of $H_1,\dots,H_n$ in which for every $i\neq j$ every vertex of $H_i$ is connected to every vertex of $H_j$. This graph is called the \emph{join} of $H_1,\dots,H_n$ 

We use the shortcut $\ominus$ for the graph with no edge of size $n$. Thus $\ominus[H_1,\dots ,H_n]$ is the graph given by the disjoint union of $H_1,\dots,H_n$ This graph is called the \emph{union} of $H_1,\dots,H_n$.

\end{nota}

This construction allows us to transform non-plane labeled trees with internal nodes decorated with graphs, $\oplus$ and $\ominus$ into graphs.

\begin{dfn}\label{d3}
Let $\mathcal{T}_0$ be the set of rooted non-plane trees whose leaves have distinct labels in $\N$ and whose internal nodes carry decorations satisfying the following constraints: 
\begin{itemize}
    \item internal nodes are decorated with $\oplus$, $\ominus$ or a graph;
    \item If a node is decorated with some graph $G$ then $|G|\geq 2$ and this node has $|G|$ children. If a node is
decorated with $\oplus$ or $\ominus$ then it has at least 2 children.
   
\end{itemize}

A tree $t\in \mathcal{T}_0$ is called a \emph{substitution tree} if the labels of its leaves are in $\{1,\dots, |t|\}$.

\end{dfn}

We call \emph{linear} the internal nodes decorated with $\oplus$ or $\ominus$ and \emph{non-linear} the other ones.

\begin{nota}
For a non-plane rooted tree $t$, and an internal node $\mathfrak{n}$ of $t$, let $t_{\mathfrak{n}}$ be the multiset of trees attached to $\mathfrak{n}$ and let $t[\mathfrak{n}]$ be the non-plane tree rooted at $\mathfrak{n}$ containing only the descendants of $\mathfrak{n}$ in $t$. 
\end{nota}

\begin{conv}

We only consider non-plane trees. However it is sometimes convenient to order the subtrees of a given node. The convention is that for some $\mathfrak{n}$ in a tree $t$ the trees of $t_{\mathfrak{n}}$ are ordered according to their minimal leaf labels.

\end{conv}

\begin{dfn}\label{d2}
Let $t$ be an element of $\mathcal{T}_0$, the weakly-labeled graph \emph{$\mathrm{Graph}(t)$} is inductively defined as follows: 
\begin{itemize}
    \item if $t$ is reduced to a single leaf labeled $j$, $\mathrm{Graph}(t)$ is the graph reduced to a single vertex labeled $j$;
    \item otherwise, the root $r$ of $t$ is decorated with a graph $H$, and $$\mathrm{Graph}(t)=H[\mathrm{Graph}(t_1),\dots, \mathrm{Graph}(t_{|H|})]$$
    where $t_i$ is the $i$-th tree of $t_r$.
\end{itemize}

If $G$ is a graph and $t_0$ is a tree in $\mathcal{T}_0$, we say that $t_0$ is a substituion tree of $G$ if $\mathrm{Graph}(t_0)=G$.
\end{dfn}

\begin{figure}[htbp]
\begin{center}
\includegraphics[scale=0.6]{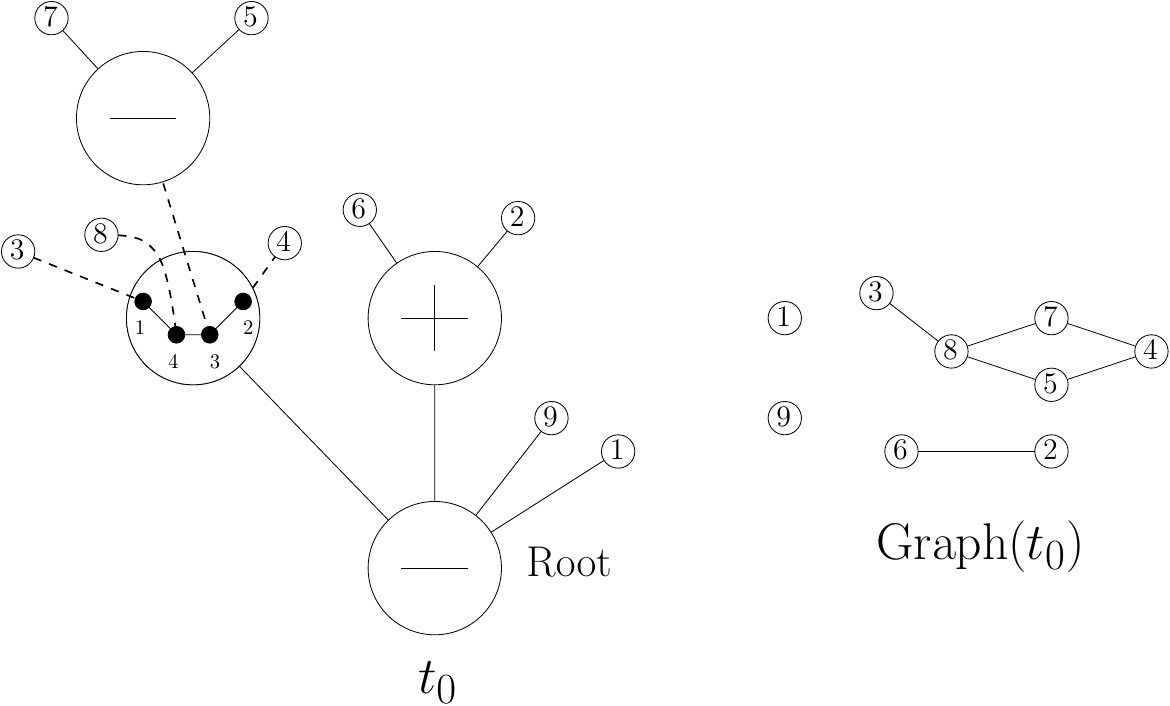}
\caption{A substitution tree $t_0$ and the corresponding graph $\mathrm{Graph}(t_0)$.}
\end{center}
\end{figure}

Note that if $t$ is a substitution tree then $\mathrm{Graph}(t)$ is a labeled graph.

\subsection{Modular decomposition}

In this short section we gather the main definitions and properties of modular decomposition. The historical reference is \cite{Gallai}, the interested reader may also look at \cite{graphclasses} or \cite{algo}.

The next definitions and theorems allow to get a unique recursive decomposition of any graph in the sense of \Cref{d2}, the modular decomposition, and to encode it by a tree.

\begin{dfn}
Let $G$ be a graph (labeled or not). A \emph{module} $M$ of $G$ is a subset of $V$ such that for every $(x,y)\in M^2$, and every $z\in V\backslash M$, $\{x,z\}\in E$ if and only if $\{y,z\}\in E$.
\end{dfn}

\begin{rem}
Note that $\emptyset, V$ and $\{v\}$ for $v\in V$ are always modules of $G$. Those sets are called the trivial modules of $G$.
\end{rem}

\begin{dfn}\label{prime}
A graph $G$ is \emph{prime} if it has at least $3$ vertices and its only modules are the trivial ones.
\end{dfn}

\begin{dfn}
A graph is called \emph{$\ominus$-indecomposable} (resp.~\emph{$\oplus$-indecomposable}) if it cannot be written as $\ominus[G_1,\dots,G_k]$ (resp.~$\oplus[G_1,\dots,G_k]$) for some $k\geq 2$ and weakly-labeled graphs $G_1,\dots, G_k$.
\end{dfn}

Note that a graph is $\ominus$-indecomposable if and only if it is connected, and $\oplus$-indecomposable if and only if its complementary is connected.

\begin{thm}[Modular decomposition, \cite{Gallai}]\label{thm1}
Let $G$ be a graph with at least $2$ vertices, there exists a unique partition $\mathcal{M}=\{M_1,\dots, M_k\}$ for some $k\geq 2$ (where the $M_i$'s are ordered by their smallest element), where each $M_i$ is a module of $G$ and such that either
\begin{itemize}
    \item $G=\oplus[M_1,\dots, M_k]$ and the $(M_i)_{1\leq i \leq k}$ are $\oplus$-indecomposable;
    \item $G=\ominus[M_1,\dots, M_k]$ and the $(M_i)_{1\leq i \leq k}$ are $\ominus$-indecomposable;
    \item there exists a unique prime graph $P$ such that $G=P[M_1,\dots, M_k]$.
\end{itemize}

\end{thm}

This decomposition can be used to encode graphs by specific trees to get a one-to-one correspondence.

\begin{dfn}
Let $t$ be a substitution tree. We say that $t$ is a \emph{modular decomposition tree} if its internal nodes are either $\oplus$, $\ominus$ or prime graphs, and if there is no child of a node decorated with $\oplus$ (resp.~$\ominus$) which is decorated with $\oplus$ (resp.~$\ominus$).

\end{dfn}

To a graph $G$ we associate a modular decomposition tree by recursively applying the decomposition of
\Cref{thm1} to the modules $(M_i)_{1\leq i \leq k}$, until they are of size $1$. First of all, at each step, we order the different modules increasingly according to their minimal vertex labels. Doing so, a labeled graph $G$ can be
encoded by a modular decomposition tree. The internal nodes are decorated with the different graphs that are encountered along the recursive decomposition process ($\oplus$ if $G=\oplus[M_1,\dots, M_k]$, $\ominus$ if $G=\ominus[M_1,\dots, M_k]$, $P$ if $G=P[M_1,\dots, M_k]$).\\
At the end, every module of size $1$ is converted into a leaf labeled by the label of the vertex.

This construction provides a one-to-one correspondence
between labeled graphs and modular decomposition trees that maps the size of a graph to the size of the corresponding tree.

\begin{prop}
Let $G$ be a graph, and $t$ its modular decomposition tree, then $t$ is the only modular decomposition tree such that $\mathrm{Graph}(t)=G$.
\end{prop}

\begin{rem}
It is crucial to consider modular decomposition trees as non-plane: otherwise, since prime graphs appearing in the decorations can have several labelings, there would be several modular decomposition trees associated with the same graph.
\end{rem}

\subsection{Expanded trees}

An important tool that will be used in \cref{lastsection} to prove \cref{thmintro2} is the notion of expanded trees, which are substitution trees corresponding to a fixed graph and maximizing the number of edges. The aim of this subsection is to prove \cref{exp}, which will be the only result used later.

\begin{dfn}
Let $G$ be a graph. An \emph{expanded tree} of $G$ is a substitution tree $t_0$ of $G$ whose non-linear nodes are decorated with prime graphs and all linear nodes have exactly two children.
\end{dfn}

We introduce the inflation operation to build expanded tree from any given tree, and to count later the number of expanded tree corresponding to a graph. Let $t$ be a substitution tree, $\mathfrak{n}$ an internal node of $t$ and $\tau$ a substitution tree of the decoration of $\mathfrak{n}$. Consider the following modifications of $t$:
\begin{itemize}
    \item the node $\mathfrak{n}$ is replaced by $\tau$;
    \item for every $j$, the $j$-th leaf of $\tau$ is replaced by the $j$-th tree of $t[\mathfrak{n}]$.
\end{itemize}

The resulting tree is called the inflation of $t$ at $\mathfrak{n}$ with $\tau$. Note that it is still a substitution tree of $\mathrm{Graph}(t)$, and that every inflation increases the number of edges if it does change $t$. See \cref{infla} for an example of inflation.

\begin{figure}
\begin{center}
\centering
\includegraphics[scale=0.55]{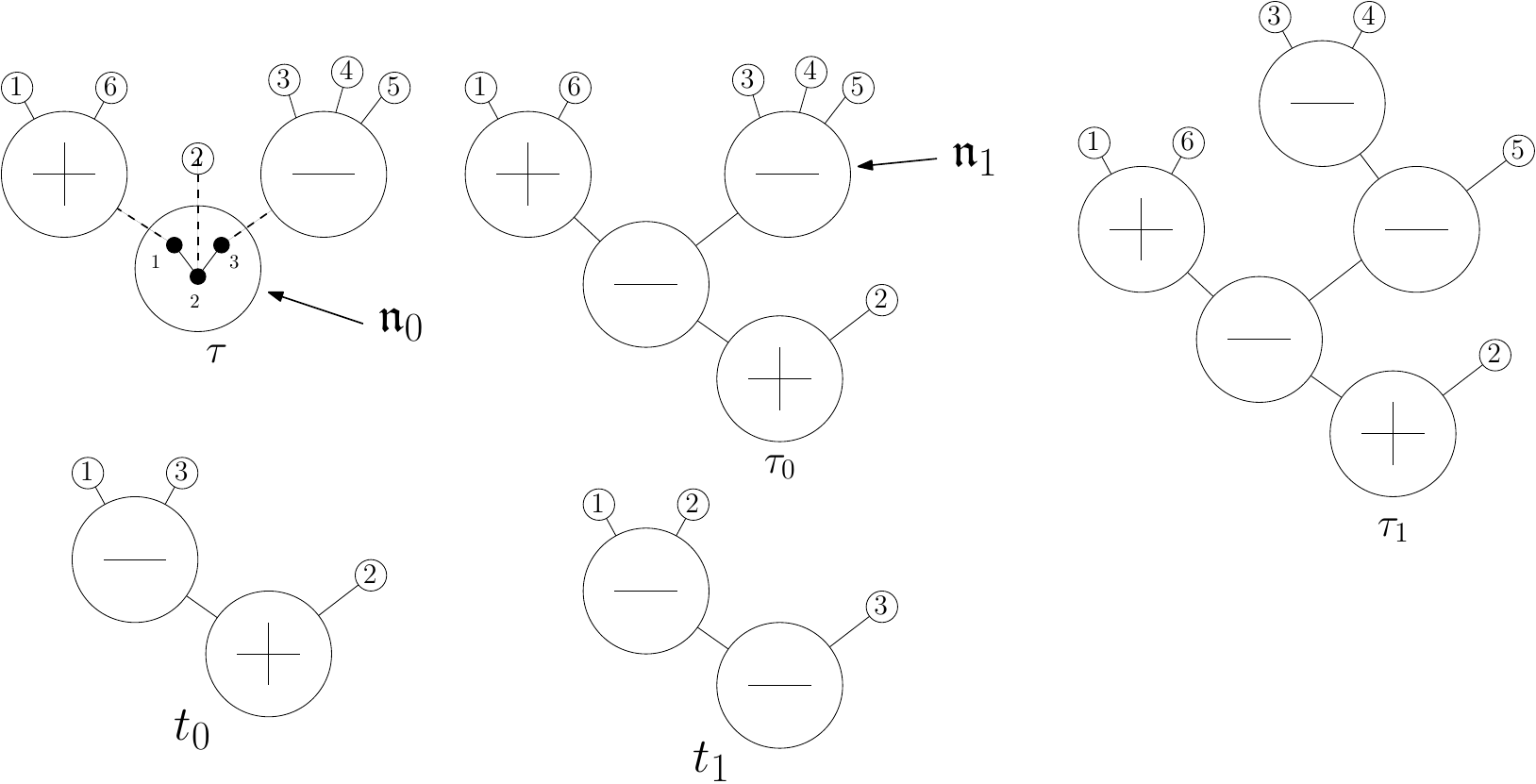}
\caption{From left to right: a tree $\tau$, $\tau_0$ the inflation of $\tau$ at $\mathfrak{n_0}$ with $t_0$, $\tau_1$ the inflation of $\tau_0$ at $\mathfrak{n_1}$ with $t_1$ which is an expanded tree.}\label{infla}
\end{center}
\end{figure}

\begin{lem}
Let $G$ be a graph and $t$ a substitution tree of $G$.
The tree $t$ can be transformed into an expanded tree of $G$ by applying successive inflations.
\end{lem}

\begin{proof}
For every non-linear node $\mathfrak{n}$ whose decoration is not prime, we can perform the inflation operation with the modular decomposition tree of the decoration of $\mathfrak{n}$ to get a tree whose non-linear nodes are all prime. Then for every linear nodes that has $k>2$ children, we can perform the inflation operation with a binary non-plane tree with $k$ leaves, and whose internal nodes are all decorated with $\oplus$. We thus get an expanded tree.
\end{proof}

\begin{lem}\label{lemexp}
Let $G$ be a graph and $t$ the modular decomposition tree of $G$. Every expanded tree of $G$ can be obtained by inflating every linear node $\mathfrak{n}$ of $t$ of size $k$ with a binary tree with $k$ leaves.
\end{lem}

\begin{proof}
Take one expanded tree $t'$ of $G$. We can merge every connected component of internal nodes of $t'$ decorated with $\oplus$ (resp. $\ominus$) into one single node decorated with $\oplus$ (resp. $\ominus$), and get a tree $\tau$ such that $\tau$ is a modular decomposition tree: since $t'$ is an expanded tree, the non-linear nodes of the resulting tree are prime, and by construction, two successive linear nodes cannot share the same decoration. Moreover $\mathrm{Graph}(t)=\mathrm{Graph}(t')=G$: thus $\tau=t$. Doing the reverse operation of merging is exactly the infation of every linear node $\mathfrak{n}$ of $t$. Since in $t'$ every linear nodes has two children, these inflations are done with binary trees. 
\end{proof}

\begin{cor}\label{exp}
Let $G$ be a graph of size $n$, $\tau$ its modular decomposition tree. Let $k$ (resp.~$j$) be the number of linear (resp. non-linear) nodes of $\tau$. We arbitrarily order the linear (resp. non-linear) nodes of $\tau$: let $d_i$ (resp.~$e_i$) be the number of children of the $i$-th linear (resp. non-linear) node of $\tau$ for $1\leq i \leq k$ (resp.~$1\leq i \leq j$). Then there are exactly $\prod\limits_{i=1}^k (2d_i-3)!!$ expanded trees of $G$, each having $2n-2-\sum\limits_{i=1}^{j}(e_i-2)$ edges.
\end{cor}

\begin{proof}

\Cref{lemexp} imply that there are exactly $$\prod\limits_{i=1}^k \#\{\text{Non-plane trees with}\ d_i\ \text{vertices}\}$$
expanded trees of $G$.

Note that unrooted non-plane trees with $r+1$ vertices are in bijection with rooted non-plane trees with $r$ vertices (by deleting the leaf of label $r+1$ and choosing the root to be its only neighbour in the tree). From \cite{aldous}, there are $(2r-3)!!$ rooted non-plane trees with $r$ vertices. Thus there exists $(2r-3)!!$ non-plane trees with leaves labeled from $1$ to $r+1$ and there are exactly $\prod\limits_{i=1}^k (2d_i-3)!!$ expanded trees of $G$.

 Note that $\#\text{edge}=\#\text{internal node}+\#\text{leave}-1$. Thus $\#\text{internal node}=\#\text{edge}-(n-1)$. Moreover, the number of edges is also the sum of the numbers of children of every internal nodes. Since all non-linear nodes are binary, $\sum\limits_{i=1}^{j}|e_i|-2=\#\text{edge}-2\#\text{internal node}$. Thus $2n-2-\sum\limits_{i=1}^{j}(e_i-2)=\#\text{edge}$.
\end{proof}

In the following, with the notations of \cref{exp}, we denote for every graph $G$
\begin{align}\label{betaG}
    \beta(G):=\frac12 \sum\limits_{i=1}^{j}(e_i-2).
\end{align}

Note that $\beta(G)$ is positive if and only if $G$ is not a cograph.

\section{Background on graphons}\label{graphon}

We now review the necessary material on graphons. We refer the reader to \cite{lovasz} for a comprehensive presentation of deterministic graphons, while \cite{janson} studies specifically the convergence of random graphs in the sense of graphons. Here we only recall the properties needed to prove the convergence of random graphs towards the Brownian cographon (see \cite{bassino2021random}).

\begin{dfn}\label{Graphon}
A graphon is an equivalence class of symmetric functions $f:[0,1]^2\mapsto [0,1]$, under the equivalence relation $\sim$, where $f\sim g$ if there exists a measurable function $\phi:[0,1]\mapsto[0,1]$ that is invertible and measure preserving such that, for almost every $(x,y)\in[0,1]^2$, $f(\phi(x),\phi(y))=g(x,y)$. We denote by $\tilde{\mathcal{W}}$ the set of graphons.
\end{dfn}

Intuitively graphons can be seen as continuous analogous of graph adjacency matrices, where graphs are considered  up to relabeling (hence the quotient by $\sim$). There is a natural way to embed a finite graph into graphons:

\begin{dfn}\label{graphe-graphon}
Let $G$ be a (random) graph of size $n$. We define the (random) graphon $W_G$ to be the equivalence class of $w_G:[0,1]^2\mapsto[0,1]$ defined by:
\begin{align*}
    \forall(x,y)\in[0,1]^2\quad w_G(x,y)&:=1_{\lceil nx\rceil \text{connected to}\lceil ny\rceil}
\end{align*}
\end{dfn}

There exists a pseudo-metric $\delta_{\square}$ on the set of graphons. If $\tilde{\mathcal{W}}$ is the set of graphons quotiented by the equivalence relation $W\equiv W'$ if $\delta_{\square}(W,W')=0$, we get that $(\tilde{\mathcal{W}},\delta_{\square})$ is compact \cite[Chapter 8]{lovasz}.

Now we introduce random variables that play the role of margins in the space of graphons. For $k\geq 1$ and $\mathbf{W}$ a random graphon, we denote by $\mathrm{Sample}_k(\mathbf{W})$ the unlabeled random graph built as follows:
$\mathrm{Sample}_k(W)$ has vertex set $\{v_1,v_2,\dots,v_k\}$ and,
letting $(X_1,\dots,X_k)$ be i.i.d.~uniform random variables in $[0,1]$,
we connect vertices $v_i$ and $v_j$ with probability $w(X_i,X_j)$
(these events being independent, conditionally on $(X_1,\cdots,X_k)$ and $\mathbf{W}$). The construction does not depend on the representation of the graphon, and is compatible with the quotient. Moreover, if $\mathbf{W}$ is a random graphon,  $(\mathrm{Sample}_k(\mathbf{W}))_{k\geq 0}$ determines the law of $\mathbf{W}$.

Since $(\tilde{\mathcal{W}},\delta_{\square})$ is compact, we can define for $\delta_{\square}$ the convergence in distribution of a random graphon. If $(\mathbf{G}^{(n)})_{n\geq 1}$ is a sequence of random graphs, there exists a simple criterion \cite[Theorem 3.1]{janson} characterizing the convergence in distribution of $(W_{\mathbf{G}^{(n)}})$ with respect to $\delta_{\square}$:

\begin{thm}[Rephrasing of \cite{janson}, Theorem $3.1$]
\label{Portmanteau}
	For any $n$, let $\mathbf{G}^{(n)}$ be a random graph of size $n$. Denote by $W_{\mathbf{G}^{(n)}}$ the random graphon associated to ${\mathbf{G}^{(n)}}$ and $\mathbf{W}$ a random graphon.
	The following assertions are equivalent:
	\begin{enumerate}
		\item [(a)] The sequence of random graphons $(W_{\mathbf{G}^{(n)}})_{n\geq 1}$ converges in distribution to $\mathbf{W}$. 
		\item [(b)] The random infinite vector $\left(\frac{\mathrm{Occ}_{H}(\mathbf{G}^{(n)})}{n^{|H|}}\right)_{H\text{ finite graph}}$ converges in distribution in the product topology to $$\left(\mathbb{P}(\mathrm{Sample}_{|H|}(\mathbf{W})=H\ \vert\  \mathbf{W})\right)_{H\text{ finite graph}}.$$ 
\item [(c)]For every finite graph $H$, there is a constant $\Delta_H \in [0,1]$ such that \[\mathbb{E}\left[\frac{\mathrm{Occ}_{H}(\mathbf{G}^{(n)})}{n^{|H|}}\right] \xrightarrow{n\to\infty} \mathbb{P}(\mathrm{Sample}_{|H|}(\mathbf{W})=H).\]

        \item [(d)] For every $\ell\geq 1$, denote by $\mathbf{\mathfrak{I}_\ell}^{(n)}$ be a uniform partial injection from $\{1,\dots,n\}$  whose image is $\{1,\dots, \ell\}$
Then the subgraph of $\mathbf{G}^{(n)}$ induced by $\mathbf{\mathfrak{I}_\ell}^{(n)}$ converges in distribution to $\mathrm{Sample}_{\ell}(\mathbf{W})$.

	\end{enumerate}

\end{thm}

For every $p\in[0,1]$, the article \cite{bassino2021random} introduces a random graphon $\mathbf{W}^{p}$ called the Brownian cographon which can be explicitly constructed as a function of a realization of a Brownian excursion. Besides, \cite[Proposition 5]{bassino2021random} states that the distribution of the Brownian cographon is characterized\footnote{This characterization is strongly linked to the remarkable property that $k$ uniform leaves in the CRT induce a uniform binary tree with $k$ leaves, see again \cite[Section 4.2]{bassino2021random}.} by the fact that for every $k\geq 2$,  $\mathrm{Sample}_k(\mathbf{W}^{p})$ has the same law as the unlabeled version of $\mathrm{Graph}(\mathbf{b}_k^p)$ with $\mathbf{b}_k^p$ a uniform labeled binary tree with $k$ leaves and i.i.d.~uniform decorations in $\{\oplus, \ominus \}$, such that the probability of an internal node being decorated with $\oplus$ is $p$.

A consequence of this characterization is a simple criterion for convergence to the Brownian cographon.

\begin{lem}[Generalization of \cite{bassino2021random} Lemma $4.4$]\label{111}
For every $n\geq 1$, let $\mathcal{T}_n$ be a subset of trees in $\mathcal{T}_0$ with $n$ leaves.  For every positive integer $n$, let $\mathbf{T}^{(n)}$ be a uniform random tree in $\mathcal{T}_n$ with $n$ vertices.
	For every positive integer $\ell$, $\mathbf{\mathfrak{I}_\ell}^{(n)}$ be a uniform partial injection from $\{1,\dots,n\}$ to $\N$ whose image is $\{1,\dots, \ell\}$ and independent of $\mathbf{T}^{(n)}$. Denote by $\mathbf{T}_{\mathbf{\mathfrak{I}_\ell}^{(n)}}^{(n)}$ the subtree induced by $\mathbf{\mathfrak{I}_\ell}^{(n)}$.

	Suppose that for every $\ell$ and for every binary non-plane tree $\tau$ with $\ell$ leaves and having $V_+$ (resp.~$V_-$) internal nodes decorated with $\oplus, $ (resp.~$\ominus$), 
	\begin{equation}\mathbb{P}(\mathbf{T}_{\mathbf{\mathfrak{I}}^{(n)}}^{(n)}=\tau) \xrightarrow[n\to\infty]{} \frac{(\ell-1)!}{(2\ell-2)!}p^{V_+}(1-p)^{V_-}.\label{eq:factorisation}\end{equation}
	Then $W_{\mathrm{Graph}(\mathbf{T}^{(n)})}$ converges as a graphon to the Brownian cographon $\mathbf{W}^{p}$ of parameter $p$.
	
\end{lem}

\begin{proof}
    The proof is very similar to the proof of Lemma $4.4$ in \cite{bassino2021random}. Let $\tau$ be a binary tree, $\ell$ its number of tree,  $V_+$ (resp.~$V_-$) its number of internal nodes decorated with $\oplus, $ (resp.~$\ominus$). Since there are $\frac{(2\ell-2)!}{((\ell-1)!}$ binary non-plane tree with $\ell$ leaves, $\mathbb{P}(b_k^{p}=\tau)= \frac{(\ell-1)!}{(2\ell-2)!}p^{V_+}(1-p)^{V_-}$. Thus we have the following convergence as random labeled graphs:

    $$\mathrm{Graph}(\mathbf{T}_{\mathbf{\mathfrak{I}}^{(n)}}^{(n)})\stackrel{(d)}{\longrightarrow}\mathrm{Graph}(b_k^{p})\ \overset{\text{(d)}}{{=}}\ \mathrm{Sample}_k(\mathbf{W}^p).$$

    After forgetting the labels, we get that criterion $(d)$ in \cref{Portmanteau} is verified for $\left(\mathrm{Graph}(\mathbf{T}^{(n)})\right)_n$, which converges to $\mathbf{W}^{p}$.
\end{proof}

\section{Graphs in $\mathcal{G}_{\mathcal{P}}$: enumerative results}\label{sec4}
The aim of this section is to compute several exponential generating series of graph classes which will be used in \cref{sec5}.
\subsection{Exact enumeration}\label{sec:exa}

Throughout this section, we consider a fixed set $\mathcal{P}$ of prime graphs stable by relabeling. Let $P(z)=\sum\limits_{G\in\mathcal{P}}\frac{z^{|G|}}{|G|!}$ be the exponential generating function associated to a given set $\mathcal{P}$.

\begin{dfn}
    Let $\mathcal{T}_{\mathcal{P}}$ be the set of modular decomposition trees whose node are either linear or in $\mathcal{P}$. Let $\mathcal{G}_{\mathcal{P}}$ be the set of graphs whose modular decomposition tree is in $\mathcal{T}_{\mathcal{P}}$.
\end{dfn}

For $n\in \N$, let $\mathcal{P}_n$ be the set of graphs $G$ in $\mathcal{P}$ of size $n$.

Let $T$ be the exponential generating function of $\mathcal{T}_{\mathcal{P}}$ counted by their number of leaves. Denote by $\mathcal{T}_{\mathrm{not}\oplus}$ the set of all $t\in \mathcal{T}_{\mathcal{P}}$ whose root is not decorated with $\oplus$ (resp.~$\ominus$) and by $T_{\mathrm{not}\oplus}$ (resp.~$T_{\mathrm{not}\ominus}$) the corresponding exponential generating function. (Regarding generating functions, we skip the dependence on $\mathcal{P}$.) Throughout section $2.1$, all generating functions are considered as formal power series.

\begin{thm}

The exponential generating function $T_{\mathrm{not}\oplus}$ verifies the following equation:
\begin{align}
  \label{ecu123}
T_{\mathrm{not}\oplus}=&z+P(\exp(T_{\mathrm{not}\oplus})-1)+\exp(T_{\mathrm{not}\oplus})-1-T_{\mathrm{not}\oplus},
\end{align}
and the series $T$ and $T_{\mathrm{not}\ominus}$ are simply given by the following equations:
\begin{align}
  \label{ecu124}
&T=\exp(T_{\mathrm{not}\oplus})-1\\
    \label{ecutri}
 &T_{\mathrm{not}\ominus}=T_{\mathrm{not}\oplus}
\end{align}

Moreover, \cref{ecu123} with $T_{\mathrm{not}\oplus}(0)=0$ determines (as a formal series) uniquely the generating function $T_{\mathrm{not}\oplus}$.
\end{thm}

\begin{proof}
Note that there is a natural involution on $\mathcal{T}_{\mathcal{P}}$: the decoration of every linear node can be changed to its opposite: $\oplus$ to $\ominus$, and $\ominus$ to $\oplus$. Therefore $T_{\mathrm{not}\oplus}=T_{\mathrm{not}\ominus}$.

First, we prove that
\begin{align}
  \label{ecu126}
T_{\mathrm{not}\oplus}&=z+P(T)+(\exp(T_{\mathrm{not}\oplus})-1-T_{\mathrm{not}\oplus}).
\end{align}

We split the enumeration of the trees $t\in \mathcal{T}_{\mathrm{not}\oplus}$ according to the different possible cases.

\begin{itemize}
    \item The tree $t$ is a single leaf (which gives the $z$ in \cref{ecu126}).
    
    \item  The tree $t$ has a root decorated with a graph $H$ belonging to $\mathcal{P}$. The exponential generating function for a fixed $H$ is $\frac{T^{|H|}}{|H|!}$. Summing over all $H$ and all $n$ gives the term $P(T)$ in \cref{ecu126}.
   
   \item The tree $t$ has a root $r$ decorated with $\ominus$ and having $k$ children with $k\geq 2$. In this case, the generating function of the set of the $k$ subtrees of $t_r$ is $\frac{T_{\mathrm{not}\oplus}^k}{k!}$. Summing over all $k$ implies that the exponential generating function of all trees in case $(D3)$ with a root decorated with $\oplus$ is $\exp(T_{\mathrm{not}\oplus})-1-T_{\mathrm{not}\oplus}$.

\end{itemize}

Summing all terms gives \cref{ecu126}.

We split the enumeration of the trees $t\in \mathcal{T}_{\mathcal{P}}$ according to the different possible cases. 

\begin{itemize}
    \item The root is not decorated with $\oplus$: the exponential generating function is $T_{\mathrm{not}\oplus}$
    \item The root is decorated with $\oplus$. The exponential generating function is $\exp(T_{\mathrm{not}\oplus})-1-T_{\mathrm{not}\oplus}$
\end{itemize}

Summing gives \cref{ecu124}. Then \cref{ecu123} is an easy consequence from \cref{ecu124,ecu126}.\smallskip

Note that \cref{ecu123} can be rewritten as:
\begin{align}
  \label{eq:equation55555}
T_{\mathrm{not}\oplus}&=z+\sum\limits_{k\geq 3}|\mathcal{P}_k|\left(\sum\limits_{\ell \geq 1}T_{\mathrm{not}\oplus}^{\ell}\right)^k+\sum\limits_{k\geq 2}\frac{T_{\mathrm{not}\oplus}^k}{k!}.
\end{align}

For every $n\geq 1$, the coefficient of degree $n$ of $T_{\mathrm{not}\oplus}$ only depends on coefficients of lower degree as $T_{\mathrm{not}\oplus}(0)=0$. Thus \cref{ecu123} combined with $T_{\mathrm{not}\oplus}(0)=0$ determines uniquely $T_{\mathrm{not}\oplus}$.\end{proof}

We are going to define the notions of trees with marked leaves, and of blossomed trees, which will be crucial in the next section. We insist on the fact that the size parameter counts the number of leaves \textbf{including the marked ones but not the blossoms}.

\begin{dfn}\label{mark}
A \emph{marked tree} is a pair $(t,\mathfrak{I})$ where $t$ is a tree and $\mathfrak{I}$ a partial injection from the set of labels of leaves of $t$ to $\N$. The number of marked leaves is the size of the domain of $\mathfrak{I}$ denoted by $|(t,\mathfrak{I})|$, and a leaf is marked if its label $j$ is in the domain, its mark being $\mathfrak{I}(j)$.
\end{dfn}

\begin{rem}
In the following, we consider marked trees $(t,\mathfrak{I})$, and subtrees $t'$ of $t$. The marked tree $(t',\mathfrak{I})$ refers to the marked tree $(t',\mathfrak{I}')$ where $\mathfrak{I}'$ is the restriction of $\mathfrak{I}$ to the set of labels of leaves of $t'$.
\end{rem}

\begin{rem}
Let $\mathcal{F}\in\{\mathcal{T}_{\mathcal{P}},\mathcal{T}_{\mathrm{not}\ominus}\}$, and $F$ be its generating exponential function. The exponential generating function of trees in $\mathcal{F}$ with a marked leaf is $zF'(z)$: if there are $f_n$ trees of size $n$ in $\mathcal{F}$, there are $nf_n$ trees with a marked leaf. Thus the generating exponential function is $\sum \limits_{n\geq 1}\frac{n f_n}{n!}z^{n}=zF'(z)$.
\end{rem}

\subsubsection*{Blossoming transformation}

Let $t$ be a tree in $\mathcal{T}_{\mathcal{P}}$and  $\ell$ a leaf of $t$. We introduce an operation called the \emph{blossoming} of $(t,\ell)$ as follows:
\begin{itemize}
\item For every non-linear node $\mathfrak{n}$ belonging to the shortest path between $\ell$ and the root, let $i$ be the only integer such that $\ell$ is in the $i$-th tree of $t_{\mathfrak{n}}$. We replace the label $i$ of the decoration of $\mathfrak{n}$ by $*$, and do the reduction on the decoration of $\mathfrak{n}$; 
\item We replace the label of $\ell$ by $*$ and do the reduction on $t$.
\end{itemize}

We extend this operation to internal node: if $\mathfrak{n}$ is a internal node, we replace $t[\mathfrak{n}]$ by its leaf of smallest label, and do the blossoming operation on the tree obtained. The resulting tree is still called the \emph{blossoming} of $t$ at $\mathfrak{n}$ (see example in \cref{transfo}.

\begin{figure}
\begin{center}
\centering
\includegraphics[scale=0.7]{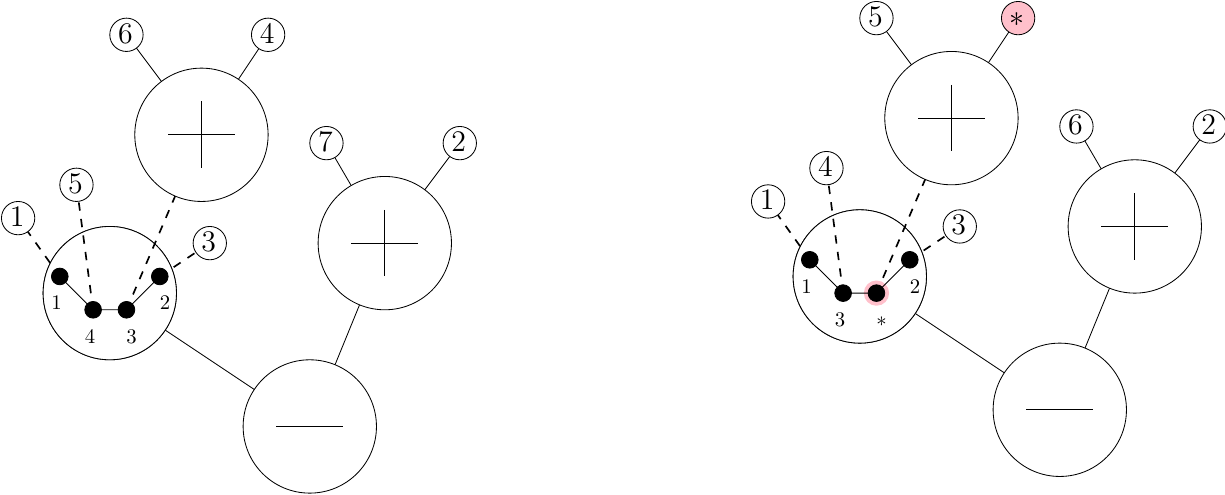}\label{transfo}
\caption{Left: a tree $t\in\mathcal{T}_{\mathcal{P}}$. Right: the blossoming of $t$ at the leaf of label $4$.}
\end{center}
\end{figure}

\begin{dfn}\label{dfnblossom}

A \emph{blossomed tree} is a tree that can be obtained by the blossoming of a tree in $\mathcal{T}_{\mathcal{P}}$. Its size is its number of leaves without blossom. 

A blossom is \emph{$\oplus$-replaceable} (resp.~\emph{$\ominus$-replaceable}) if its parent is not decorated with $\oplus$ (resp.~$\ominus$).

\end{dfn}

\begin{rem}
Similarly to a tree, a blossomed tree can be marked by a partial injection $\mathfrak{I}$.
\end{rem}

We denote $\mathcal{T}^b$ and $\mathcal{T}_a^b$ with $a\in\{\mathrm{not}\oplus,\mathrm{not}\ominus\}$, and $b\in\{\oplus,\ominus,\mathsf{blo}\}$ the set of trees whose root is not $\oplus$ (resp.~$\ominus$) if $a=\mathrm{not}\oplus$ (resp.~$a=\mathrm{not}\ominus$), and with one blossom that is $b$-replaceable if $b=\oplus$ or $\ominus$, or just with one blossom if $b=\mathsf{blo}$.

We define $T^b$ and $T_a^b$ to be the corresponding exponential generating functions of trees, counted by the number of non blossomed leaves.

However, we take the convention that $T_{\mathrm{not}\oplus}^{\oplus}(0)=0=T_{\mathrm{not}\ominus}^{\ominus}$. In other words, a single leaf is neither in $\mathcal{T}_{\mathrm{not}\oplus}^{\oplus}$ nor in $\mathcal{T}_{\mathrm{not}\ominus}^{\ominus}$. The other series have constant coefficient $1$.

\begin{rem}
From the previously defined involution, it follows that $T_{\mathrm{not}\oplus}^{\ominus}=T_{\mathrm{not}\ominus}^{\oplus}$, $T_{\mathrm{not}\oplus}^{\oplus}=T_{\mathrm{not}\ominus}^{\ominus}$ et $T^{\oplus}=T^{\ominus}$ and $T_{\mathrm{not}\oplus}^{\mathrm{blo}}=T_{\mathrm{not}\ominus}^{\mathrm{blo}}$.
\end{rem}

\begin{prop}\label{4.4}
The functions $T^{\mathrm{blo}}, T_{\mathrm{not}\oplus}^{\mathrm{blo}}$ are given by the following equations:

\begin{align}
\label{o1}
&T^{\mathrm{blo}}=T'\\
\label{o2}
&T_{\mathrm{not}\oplus}^{\mathrm{blo}}=T_{\mathrm{not}\oplus}'=T'\exp(-T_{\mathrm{not}\oplus})
\end{align}

\end{prop}

\begin{proof}
Let $\mathcal{F}\in\{\mathcal{T}_{\mathcal{P}},\mathcal{T}_{\mathrm{not}\ominus}\}$, and $F$ be its generating exponential function. The exponential generating function of blossomed trees in $\mathcal{F}$ is $F'(z)$: if there are $f_n$ trees of size $n$ in $\mathcal{F}$, there are $nf_n$ trees with a marked leaf. Since $\mathcal{P}$ is stable by relabeling, every blossomed tree of size $n-1$ can be obtained from exactly $n$ trees with a marked leaf (after blossoming the marked leaf). Thus there are $f_n$ blossomed trees of size $n-1$, the generating exponential function is $\sum \limits_{n\geq 1}\frac{f_n}{(n-1)!}z^{n-1}=F'(z)$ which implies \cref{o1,o2} (the last equality of \cref{o2} comes from differentiating \cref{ecu124}).
\end{proof}

\begin{thm}\label{4.5}
The functions $T^{\oplus}, T_{\mathrm{not}\oplus}^{\oplus}, T_{\mathrm{not}\ominus}^{\oplus}$ are given by the following equations: 

\begin{align}
\label{b1}
&T^{\oplus}=T'\exp(-T_{\mathrm{not}\oplus})\\
\label{b2}
&T_{\mathrm{not}\oplus}^{\ominus}=\frac{1}{\exp(T_{\mathrm{not}\oplus})}T^{\oplus}\\
\label{b3}
&T_{\mathrm{not}\oplus}^{\oplus}=\frac{T^{\oplus}-1}{\exp(T_{\mathrm{not}\oplus})}
\end{align}

\end{thm}

\begin{proof}

    First let's prove that $T'=T^{\oplus}\exp(T_{\mathrm{not}\oplus})$ using that $T'$ is the generating function of blossomed trees.
    \begin{figure}[htbp]
\begin{center}
\centering
\includegraphics[scale=0.6]{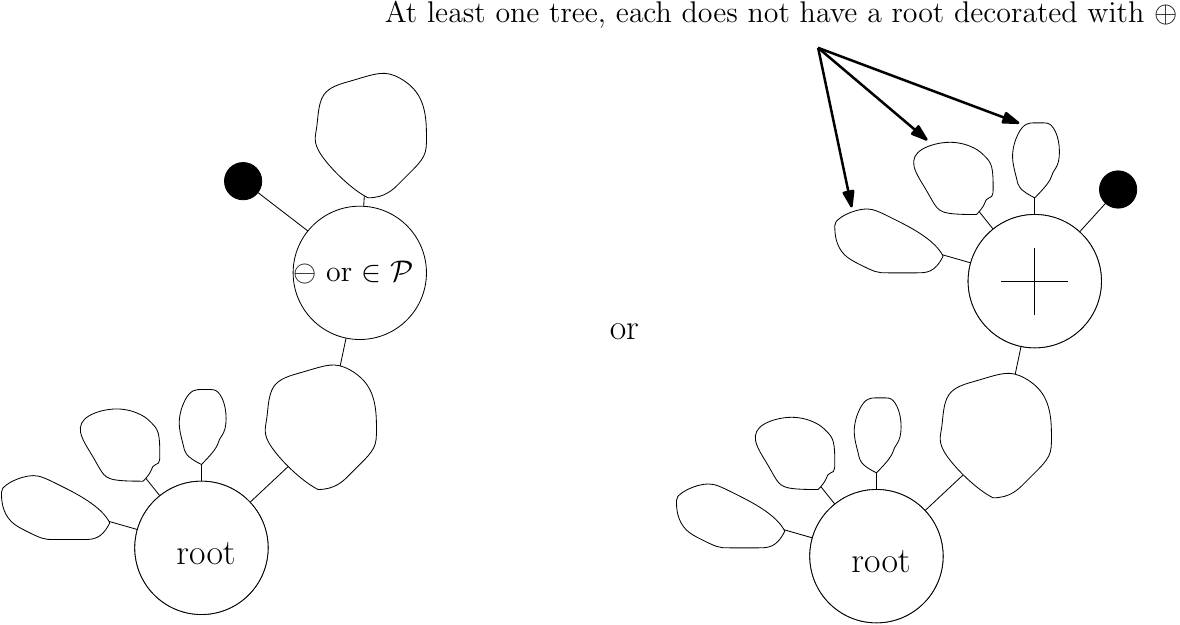}
\caption{Illustration of the proof of \cref{b1}.}
\end{center}
\end{figure}
    A blossomed tree $t$ is in exactly one of both cases:
    \begin{itemize}
        \item the blossom of $t$ is $\oplus$-replaceable, thus the exponential generating series is $T^{\oplus}$;
        \item the blossom of $t$ is not $\oplus$-replaceable, let $n$ be its parent. We define $t'$ to be $t$ blossomed at $n$. The exponential generating function of $t$ is the product of the exponential generating function of $t'$, which is $T^{\oplus}$ as the blossom must be $\oplus$ replaceable, and of the exponential generating function of $t_{n}$ (a forest of trees whose root is not decorated with $\oplus$, with at least one tree) which is $\exp(T_{\mathrm{not}\oplus})-1$. Thus the exponential generating function of blossomed trees whose blossom is not $\oplus$-replaceable is $T^{\oplus}(\exp(T_{\mathrm{not}\oplus})-1)$
    \end{itemize}

    Summing implies $T'=T^{\oplus}\exp(T_{\mathrm{not}\oplus})$ which is equivalent to \cref{b1}.

\begin{figure}[htbp]
\begin{center}
\centering
\includegraphics[scale=0.6]{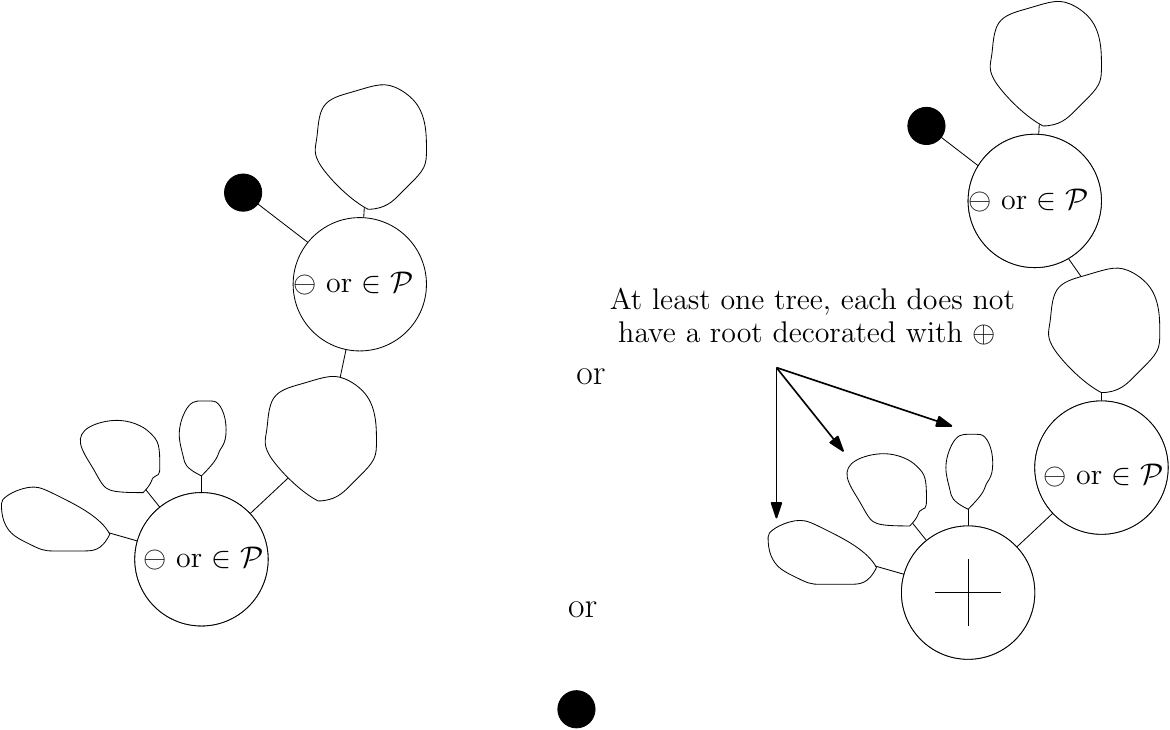}
\caption{Illustration of the proof of \cref{b2}.}
\end{center}
\end{figure}
Since the proofs of \cref{b2} and \cref{b3} are very similar, we only prove \cref{b3}. Let $t$ be a tree in $\mathcal{T}^{\oplus}$: 
\begin{itemize}
    \item The tree $t$ can be a blossom, whose exponential generating function is $1$;
    \item The tree $t$ can have a root not decorated with $\oplus$, the corresponding exponential generating function is $T_{\mathrm{not}\oplus}^{\oplus}$
    \item The tree $t$ can have a root $r$ decorated with $\oplus$ and having $k$ children with $k\geq 2$. There are $k-1$ subtrees without blossom, and $1$ with a blossom. Thus the generating function of the set of the $k$ subtrees of $t_r$ is $\frac{T_{\mathrm{not}\oplus}^{k-1}}{(k-1)!}T_{\mathrm{not}\oplus}^{\oplus}$. Summing over all $k$ gives that the exponential generating function of all trees in case $(D3)$ with a root decorated with $\ominus$ is $$\sum \limits_{k\geq 2} \frac{T_{\mathrm{not}\oplus}^{k-1}}{(k-1)!}T_{\mathrm{not}\oplus}^{\oplus}=(\exp(T_{\mathrm{not}\oplus})-1)T_{\mathrm{not}\oplus}^{\oplus}$$
\end{itemize}
Summing gives $T^{\oplus}=1+\exp(T_{\mathrm{not}\oplus})T_{\mathrm{not}\oplus}^{\oplus}$ which implies \cref{b3}.
\end{proof}

\subsection{Asymptotic enumeration}\label{sec4.2}
In the following, we derive from the previously obtained equations the radii of the different series introduced, the asymptotic behavior of the different series in $R$ and an equivalent of the number of graphs in $\mathcal{G}_{\mathcal{P}}$

From now on, we assume that $P$ has a positive radius of convergence $R_0$.
\label{pageP}
Denote by $P(R_0)$ the limit in $[0,+\infty]$ of $P$  at $R_0^-$ (which exists as $H$ has nonnegative coefficients and radius of convergence at least $R_0$). Define $\Lambda(w)=P(\exp(w)-1)+\exp(w)-1-w$

\begin{lem}
    The radius of convergence of $\Lambda$ is $R_{\Lambda}:=\log(1+R_0)$.
\end{lem}

\begin{proof}
$\Lambda$ has positive coefficients and no singularity in $[0,\log(1+R_0))$ as $P$ has no singularity in $[0,R_0)$, thus it has radius of convergence at least $\log(1+R_0)$. If $R_0=+\infty$, it implies the lemma. Otherwise, by contradiction assume that the radius of convergence of $\Lambda$ is greater than $\log(1+R_0)$. The same holds for $P(\exp(w)-1)$ and $w\mapsto \exp(w)-1$ can be inverted on a neighbourhood of $\log(1+R_0)$ as its derivative is non zero, thus $S$ can be extended to a neighbourhood of $R_0$, which contradicts Pringsheim's lemma \cite[Theorem IV.6, p. 240]{flajolet2009analytic}.
\end{proof}

Denote by $\Lambda'(R_{\lambda})$ the limit in $[0,+\infty]$ of $\Lambda'$ in $R_{\Lambda}^-$. In the rest of this paper, we assume that Condition (C) defined \pageref{condi} is verified. Note that conditon (C) can be written as:

\begin{cond}\hspace{3,85cm}$\Lambda'(R_{\Lambda})>1$
\ \end{cond}

Denote by $\kappa$ the only solution in $[0,R_{\Lambda})$  of the equation: \begin{align}
\label{eqimp}
    \Lambda'(\kappa)=1
    \end{align}

Let $K:=\exp(\kappa)-1$, the equation $\Lambda'(\kappa)=1$ can be rewritten as follows:
\begin{align}
\label{k}
(1+K)(P'(K)+1)=2
\end{align}

Recall that a formal series $A$ is \emph{aperiodic} if there does not exist two integers $r\geq 0$ and $d\geq 2$ and $B$ a formal series such that $A(z)=z^rB(z^d)$.
\begin{lem}
The functions $T$, $T_{\mathrm{not}\oplus}$, $T^{\oplus}$, $T_{\mathrm{not}\oplus}^{\ominus}$, $T_{\mathrm{not}\oplus}^{\oplus}$, $T^{\mathrm{blo}}$, $T_{\mathrm{not}\oplus}^{\mathrm{blo}}$ are aperiodic.
\end{lem}

\begin{proof}

One can easily check that for each of the previous series, the coefficients of degree $3$ and $4$ are positive, and thus all the series are aperiodic.\end{proof}

\begin{dfn}
A set $\Delta$ is a $\Delta$-domain at $1$ if there exist two positive numbers $R$ and $\frac{\pi}{2}<\phi< \pi$ such that 

$$\Delta=\{z\in\C | |z|\leq R, z\neq 1, |\mathrm{arg}(1-z)|< \phi\}$$

For every $w\in \C^*$, a set is a $\Delta$-domain at $w$ if it is the image of a $\Delta$-domain by the mapping $z\mapsto zw$.
\end{dfn}

\begin{dfn}
A power series $U$ is said to be \emph{$\Delta$-analytic} if it has a positive radius of convergence $\rho$ and there exists a $\Delta$-domain $D$ at $\rho$ such that $U$ has an analytic continuation on $D$.
\end{dfn}

\begin{prop}\label{thm5}
Under Condition (C) defined p.\pageref{condi}, both $T$ and $T_{\mathrm{not}\oplus}$ have $R:=\kappa-\Lambda(\kappa)$ as radius of convergence and a unique dominant singularity at $R$. They are $\Delta$-analytic. Their asymptotic expansions near $R$ are: 

\begin{align}
&T_{\mathrm{not}\oplus}(z)=\kappa-\frac{2R}{\mu} \sqrt{1-\frac{z}{R}}+O\left(1-\dfrac{z}{R}\right)\\
\label{ecut}
&T(z)=K-(1+K)\frac{2R}{\mu}\sqrt{1-\frac{z}{R}}+O\left(1-\dfrac{z}{R}\right)\\
\label{toplus}
&T^{\oplus}=\frac{1}{\mu}\left(1-\frac{z}{R}\right)^{-\frac{1}{2}}+O\left(1\right)
\end{align}

where $\mu$ and $K$ are the constants given by:
\begin{align}\label{mu}
    \mu&=\sqrt{2R\Lambda''(\kappa)}\\
    \label{K}K&=\exp(\kappa)-1
\end{align}
\end{prop}

\begin{proof}
First of all, note that since $\Lambda'<1$ on $[0,\kappa)$, $R$ is positive

We begin with the expansion of $T_{\mathrm{not}\oplus}$ for which we apply the smooth implicit theorem \cite[Theorem VII.3, p.467]{flajolet2009analytic}. Following \cite[Sec VII.4.1]{flajolet2009analytic} we claim that $T_{\mathrm{not}\oplus}$ satisfies the settings of the so-called \emph{smooth implicit-function schema}: $T_{\mathrm{not}\oplus}$ is solution of
$$T=G(z,T),$$
where $G(z,w)=z+P(\exp(w)-1)+(\exp(w)-1-w)=z+\Lambda(w)$.

The singularity analysis of $T_{\mathrm{not}\oplus}$ goes through the study of the characteristic system:
$$\begin{cases}
&G(r,s)=s,\\
&G_w(r,s)=1
\end{cases}\qquad \text{ with }0<r<R, \ s>0
$$
where $F_x=\frac{\partial F}{\partial x}$.\\
Note that $(r,s)=\left(R,\kappa \right)$ is a solution of the characteristic system of $G$ since 
\begin{itemize}
\item $G_w(r,s)=\Lambda'(\kappa)=1$
\item $G(r,s)=R+\Lambda(\kappa)=\kappa=s$
\end{itemize}
Moreover
\begin{itemize}
    \item $G_z(r,s)=1$
    \item $G_{w,w}(r,s)=\Lambda''(\kappa)$
\end{itemize}

The expansion of $T$ is then a consequence of \cref{ecu124} and of the expansion of $T_{\mathrm{not}\oplus}$.

For the expansion of $T^{\oplus}$, since $T$ and $T_{\mathrm{not}\oplus}$ can be extended to a $\Delta$-domain at $R$, singular differentiation \cite[Theorem VI.8 p.419]{flajolet2009analytic} yields the announced expansions when $z$ tends to $R$ and that $T^{\oplus}$ can be extended to a $\Delta$-domain at $R$. These expansions show that $T^{\oplus}$ a radius of convergence exactly equal to $R$.\end{proof}

Applying the Transfer Theorem \cite[Corollary VI.1 p.392]{flajolet2009analytic} to the results of \cref{thm5}, we obtain an equivalent of the number of trees of size $n$ in $\mathcal{T}_{\mathcal{P}}$. Since there is a one-to-one correspondence between graphs in $\mathcal{G}_{\mathcal{P}}$ and trees in $\mathcal{T}_{\mathcal{P}}$, we get the following result:

\begin{thm}\label{cor1}
Under Condition (C) defined p.\pageref{condi}, the number of graphs in $\mathcal{G}_{\mathcal{P}}$ of size $n$ is asymptotically equivalent to 
$$C \frac{n!}{R^nn^{\frac{3}{2}}}\quad \text{where}\quad C=\frac{(1+K)R}{\mu\sqrt{\pi}}.$$
\end{thm}

\section{Graphs in $\mathcal{G}_{\mathcal{P}}$: enumeration of graphs with a given induced subgraph}\label{sec5}

The heart of this section is \cref{thmcoeur}, the combinatorial decomposition of graphs in $\mathcal{G}_{\mathcal{P}}$ with a given induced subgraph.  

\subsection{Induced subtrees and subgraphs}
We recall that the size of a graph is its number of vertices, and the size of a tree is its number of leaves.

\begin{dfn}[First common ancestor]
Let $t$ be a rooted tree and let $\ell_1,\ell_2$ be two distinct leaves of $t$. The \emph{first common ancestor} of $\ell_1$ and $\ell_2$ is the internal node of $t$ that is the furthest from the root and that belongs to the shortest path from the root to $\ell_1$, and the shortest path from the root to $\ell_2$.
\end{dfn}

The notations of the next definition are illustrated on \cref{dessininduit}.
\begin{dfn}[Induced subtree]\label{defg}
Let $(t,\mathfrak{I})$ be a marked tree in $\mathcal{T}_0$ ($\mathcal{T}_0$ is defined in \cref{d3}, and the notion of marked tree in \cref{mark}). The \emph{induced subtree $t_\mathfrak{I}$} of $t$ induced by $\mathfrak{I}$ is defined as:
\begin{itemize}
\item The leaves of $t_\mathfrak{I}$ are the leaves of $t$ that are marked. For every such leaf labeled with an integer $\ell$, the new label of $\ell$ is $\mathfrak{I}(\ell)$;

\item The internal nodes of $t_\mathfrak{I}$ are the internal nodes of $t$ that are first common ancestors of two or more leaves of $t_\mathfrak{I}$;

\item The ancestor-descendent relation in $t_\mathfrak{I}$ is inherited from the one in $t$;

\item For every internal node $\mathfrak{n}$ of $t$ that appears in $t_\mathfrak{I}$, let $H$ be its decoration in $t$. Set
$$J=\left\{k;\ k\text{-th tree of $t_{\mathfrak{n}}$ contains a marked leaf}\right\}$$
For $k$ in $J$, set

$$\mathfrak{L}(k)=\min \left\{\mathfrak{J}(\ell);\ \ell\ \text{marked leaf in the}\ k\text{-th tree of $\mathfrak{L}(k)$} \right\}.$$
The decoration of $\mathfrak{n}$ in $t_\mathfrak{I}$ is the reduction of $H_{\mathfrak{L}}$ as defined in \cref{d12}.

\end{itemize}

For every internal node $\mathfrak{n}$ (resp.~leaf $\ell)$ of $t_\mathfrak{I}$, we also define $\phi(\mathfrak{n})$ to be the only internal node (resp.~leaf) of $t$ corresponding to $\mathfrak{n}$.
\end{dfn}

\begin{rem}
When $(t,\mathfrak{I})$ is a marked tree and $t'$ is a subtree of $t$, we denote $t'_\mathfrak{I}$ the tree induced by the restriction of $\mathfrak{I}$ to the set of labels of leaves of $t'$.
\end{rem}

As a consequence of Definitions \ref{d12} and \ref{defg}, we obtain:

\begin{lem}\label{çasevoit}
Let $(t,\mathfrak{I})$ be a marked tree in $\mathcal{T}_0$. Then $$\mathrm{Graph}(t)_{\mathfrak{I}}=\mathrm{Graph}(t_{\mathfrak{I}}).$$
\end{lem}

\begin{figure}[htbp]
\begin{center}
\centering
\includegraphics[scale=0.8]{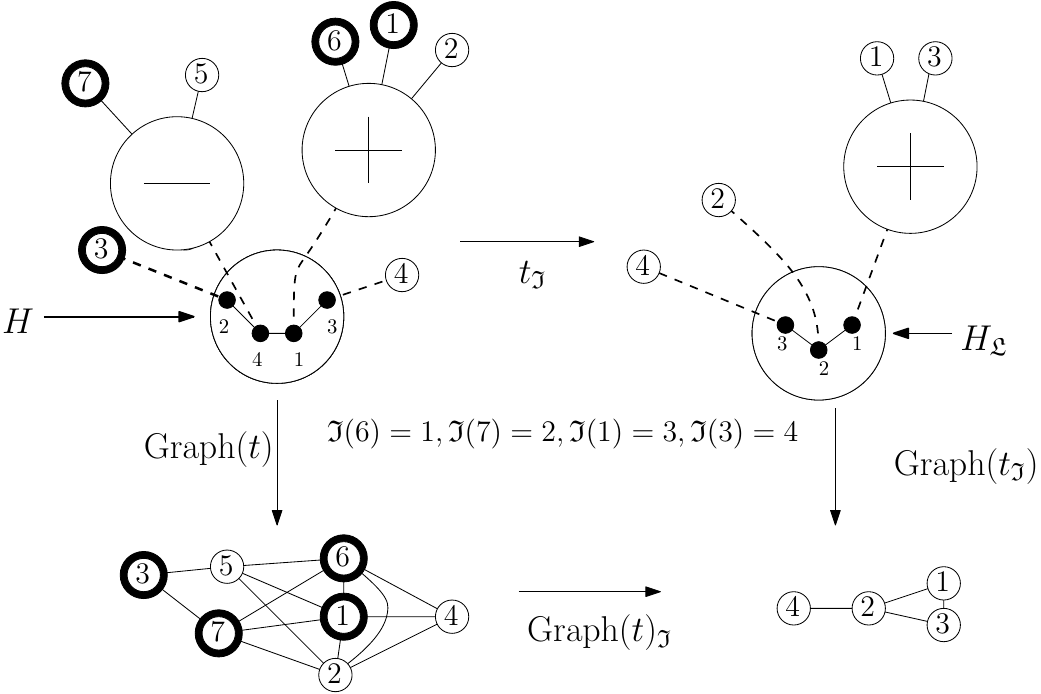}
\caption{Illustration of \cref{çasevoit}: relations between induced subgraph and induced subtree.}\label{dessininduit}
\end{center}
\end{figure}

\begin{figure}[htbp]
\begin{center}
\centering
\includegraphics[scale=0.6]{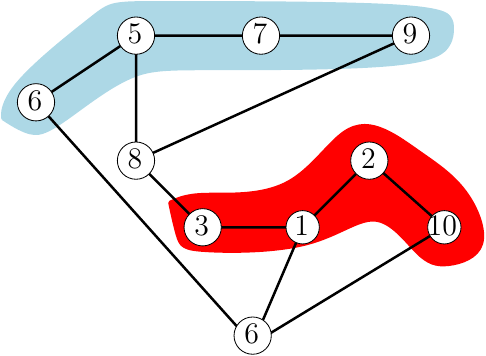}
\caption{Two occurrences of a $P_4$ (line of size $4$) in a blossomed graph $H$.}
\end{center}
\end{figure}

\begin{dfn}
For every graph $G$ without blossom, set:
\begin{align*}
&\mathrm{Occ}_{G, \mathcal{P}}(z):=\sum\limits_{H\in \mathcal{P}}\frac{\mathrm{Occ}_{G}(H)z^{|H|-|G|}}{|H|!}
\end{align*}

\end{dfn}

\begin{prop}
For every $k\geq 1$:
\begin{align}
  \label{e1}
&\sum\limits_{G: \ |G|=k}\mathrm{Occ}_{G,\mathcal{P}}(z)=P^{(k)}(z)
\end{align}

Thus for every graph $G$, $\mathrm{Occ}_{G,\mathcal{P}}$, has a radius of convergence strictly greater than $R$, the radius of convergence of $T$.
\end{prop}

\begin{proof}
Let $H$ be an element of $\mathcal{P}$. Since there are $\frac{|H|!}{(|H|-k)!}$ choices of partial injection whose image is $\{1,\dots, k\}$, we have:
$$\sum\limits_{G:\ |G|=k}\mathrm{Occ}_{G,\mathcal{P}}(z)=\sum \limits_{H\in \mathcal{P}}\sum\limits_{G: \ |G|=k}\frac{\mathrm{Occ}_{G}(H)z^{|H|-k}}{|H|!}=\sum \limits_{H\in \mathcal{P}}\frac{z^{|H|-k}}{(|H|-k)!}=P^{(k)}(z)$$

For every graph $G$, $\mathrm{Occ}_{G,\mathcal{P}}$ has nonnegative coefficients and for every $k\geq 0$, as mentioned in \Cref{sec4.2}, $P^{(k)}$ has a radius of convergence at least $R_0$, the radius of convergence of $P$, which is greater than $R$. This implies that $\mathrm{Occ}_{G,\mathcal{P}}$ has a radius of convergence greater than $R$.\end{proof}

\subsection{Enumerations of trees with a given induced subtree}

The key step in the proof of our main theorem is to compute the limiting probability (when $n\to \infty$) that a uniform induced subtree of a uniform tree in $\mathcal{T}_{\mathcal{P}}$ with $n$ leaves is a given substitution tree. 

In the following, let $\tau\in \mathcal{T}_0$ be a fixed substitution tree of size at least $2$.

\begin{dfn}
We define \emph{$\mathcal{T}_{\tau}$} to be the set of marked trees $(t,\mathfrak{I})$ where $t\in \mathcal{T}_{\mathcal{P}}$ and such that $t_\mathfrak{I}$ is isomorphic to $\tau$. We also define $T_{\tau}$ to be the corresponding exponential generating function (where the size parameter is the total number of leaves, including the marked ones).
\end{dfn}
 The aim now is to decompose a tree admitting $\tau$ as a subtree in smaller trees. Let $(t,\mathfrak{I})$ be in $\mathcal{T}_{\tau}$, note that the image of $\mathfrak{I}$ is $\{1,\dots,|\tau|\}$. A prime node $\mathfrak{n}$ of $\tau$ is such that $\phi(\mathfrak{n})$ must be a prime node. In constrast, knowing that an internal node $\mathfrak{n}'$ of $\tau$ is decorated with $\oplus$ or $\ominus$ does not give any information about the decoration of $\phi(\mathfrak{n}')$.

In order to state \cref{thmcoeur} below, we need to partition the internal nodes of $\tau$:

\begin{dfn}
Let $(t,\mathfrak{I})$ be in $\mathcal{T}_{\tau}$. We denote by $\mathcal{N}(t,\mathfrak{I})$ the set of internal nodes $\mathfrak{n}$ of $\tau$ such that $\phi(\mathfrak{n})$ is non-linear.
\end{dfn}

Note that the set of non-linear nodes of $\tau$ must be included in $\mathcal{N}(t,\mathfrak{I})$.

\begin{thm}\label{thmcoeur}
Let $\tau$ be a substitution tree of size at least $2$. Let $\Vs$ be a set of internal nodes of $\tau$ that contains every non-linear node of $\tau$ and $\overline{\Vs}$ the set of internal nodes of $\tau$ not in $\Vs$. 

Let $\mathcal{T}_{\tau,\Vs}$ be the set of marked trees $(t,\mathfrak{I})$ in $\mathcal{T}_{\tau}$ such that $\mathcal{N}(t,\mathfrak{I})=\Vs$, and let $T_{\tau,\Vs}$ be its exponential generating function.

Then
\begin{align}\label{eqlongue}
T_{\tau,\Vs}=z^{|\tau|}&T^{\mathrm{root}}\left(T_{\mathrm{not}\oplus}^{\oplus}\right)^{d_{=}}\left(T_{\mathrm{not}\oplus}^{\ominus}\right)^{d_{\neq }}\left(T_{\mathrm{not}\oplus}^{\mathsf{blo}}\right)^{d_{\overline{\Vs}\to\Vs}}\left(T_{\mathrm{not}\oplus}^{'}\right)^{d_{\overline{\Vs}\to \ell}}\exp(n_L T_{\mathrm{not}\oplus})\\
&\times (T^{\oplus})^{d_{\Vs\to\overline{\Vs}}}(T^{\mathsf{blo}})^{d_{\Vs \to \Vs}}(T')^{d_{\Vs\to \ell}}\prod\limits_{\mathfrak{n}\in \Vs}\mathrm{Occ}_{\mathrm{dec}(\mathfrak{n}),\mathcal{P}}(T)
\end{align}
and:
\begin{itemize}
    \item $d_{=}$ is the number of edges between two internal nodes not in \Vsr with the same decoration ($\oplus$ and $\oplus$, or $\ominus$ and $\ominus$);
    \item $d_{\neq }$ is the number of edges between two internal nodes not in \Vsr decorated with different decorations ($\oplus$ and $\ominus$);
    \item $d_{\overline{\Vs}\to\Vs}$ is the number of edges between an internal node not belonging to \Vsr and one of its children belonging to \Vsr;
    \item $d_{\Vs\to\overline{\Vs}}$ is the number of edges between an internal node belonging to \Vsr and one of its children not belonging to \Vsr;
    \item $d_{\Vs\to\Vs}$ is the number of edges between an internal node belonging to \Vsr and one of its children belonging to \Vsr;
    \item $d_{\overline{\Vs}\to \ell}$ is the number of edges between an internal node not in \Vsr and a leaf;
    \item $d_{\Vs\to \ell}$ is the number of edges between an internal node not in \Vsr and a leaf;
    \item $n_L$ is the number of internal nodes not in \Vsr;
    \item $\mathrm{dec}(\mathfrak{n})$ is the decoration of $\mathfrak{n}$;
    \item $T^{\mathrm{root}}=T^{\oplus}$ if the root of $\tau$ is not in \Vsr, $T^{\mathrm{root}}=T^{\mathsf{blo}}$ otherwise.
\end{itemize}

\end{thm}

\begin{figure}
\begin{center}
\centering
\includegraphics[width=9cm]{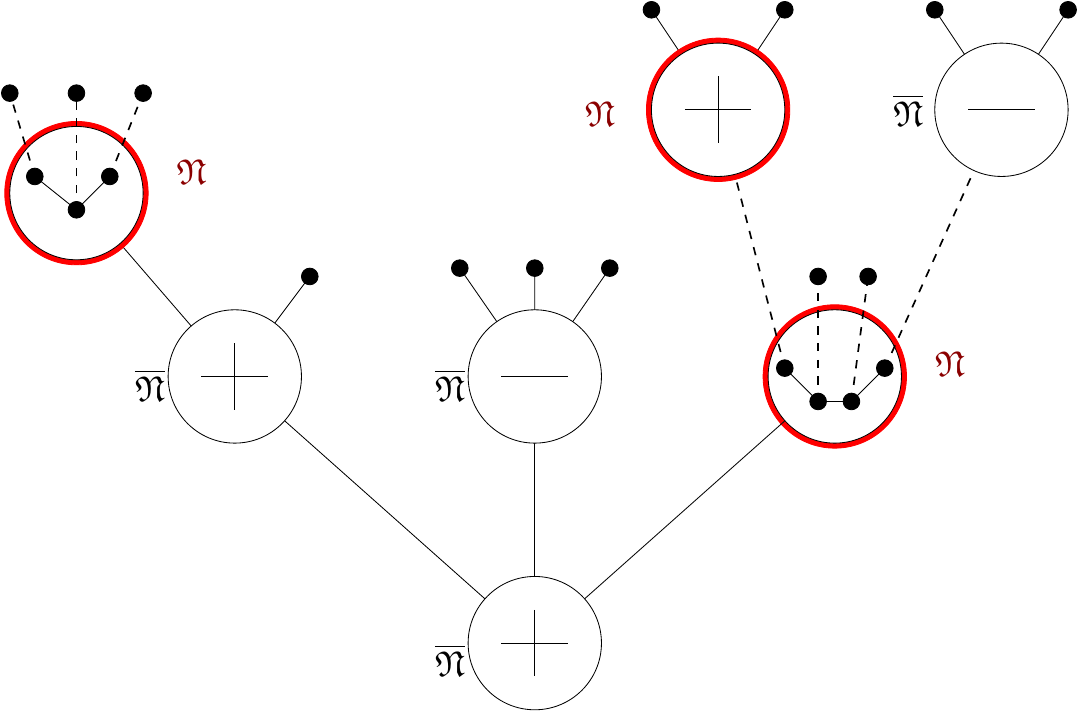}
\caption{A substitution tree $\tau$. Vertices in \Vsr are surrounded in red.}\label{fig13}
\end{center}
\end{figure}

\begin{figure}[htbp]
\begin{center}
\centering
\includegraphics[width=16cm]{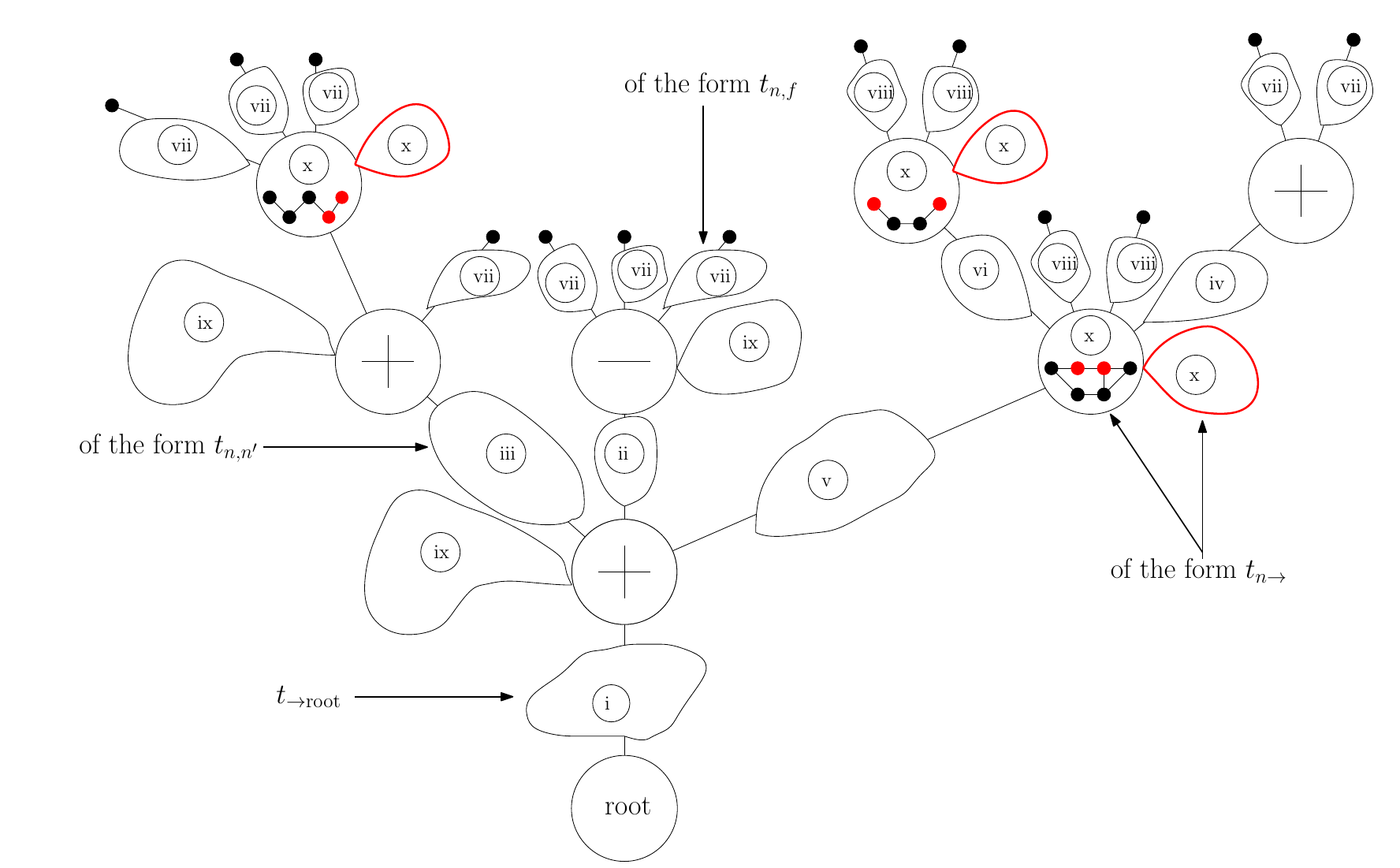}
\caption{The decomposition of a tree admitting the graph $\tau$ of \cref{fig13} as an induced tree. Roman numerals correspond to the different cases of the proof of \Cref{thmcoeur}.}
\end{center}
\end{figure}

\begin{proof}
Let $t$ be a tree in $\mathcal{T}_{\tau,\Vs}$. The proof consists in decomposing $t$ into several disjoint (blossomed) subtrees ($t_{\to \mathrm{root}},t_{\mathfrak{n}\to}, t_{\mathfrak{n}\to \mathfrak{n}'}, t_{\mathfrak{n}\to f}$) of various kinds (those who arise in the RHS of formula \cref{eqlongue}). All these subtrees are meant to be glued at blossoms, in order to recover the initial $t$. (Thus they are not counted in the generating function, in order to avoid counting them twice.)  In the following, every defined tree is assumed to be reduced.

\begin{itemize}
\item We define $t_{\to\mathrm{root}}$ to be the tree $t$ blossomed at $\phi(r_0)$, where $r_0$ is the root of $\tau$.

\item For each internal node $\mathfrak{n}$, we define $t_{\mathfrak{n}\to}$ in the following way: 

\begin{itemize}
    \item  If $\mathfrak{n}$ is not in \Vsr,  $t_{\mathfrak{n}\to}$ is the subtree of $t$ containing $\phi(\mathfrak{n})$ and all the trees of $t_{\phi(\mathfrak{n})}$ that do not contain a marked leaf of $t$.
    \item If $\mathfrak{n}$ is in $\Vs$, $t_{\mathfrak{n}\to}$ is given by the forest and the graph defined as follows (see \cref{Cas(x)}):
    \begin{itemize}
        \item The forest is composed of all subtree of $t_{\phi(\mathfrak{n})}$ which do not contain any marked leaf;
    \item  The graph is $G$ the decoration of $\phi(\mathfrak{n})$, blossomed sucessively at $v_1,\dots, v_k$, where $v_1,\dots, v_k$ are defined as follows. Let $t_{\mathfrak{n},1},\dots, t_{\mathfrak{n},k}$ be the trees in $t_{\mathfrak{n}}$ with a marked leaf, ordered by \emph{minimal mark label}. For each $j\in\{1,\dots,k\}$, we define $v_j$ the vertex in $G$ of label the rank of $ t_{\mathfrak{n},j}$ in  $t_{\mathfrak{n}}$ \emph{by minimal leaf label}.
    \end{itemize}
\end{itemize}

\item For every internal nodes $\mathfrak{n},\mathfrak{n}'$ in $\tau$ such that $\mathfrak{n}'$ is a child of $\mathfrak{n}$, let $t_{\mathfrak{n}\to \mathfrak{n}'}$ be the unique tree of $t_{\phi(\mathfrak{n})}$ containing $\phi(\mathfrak{n}')$, blossomed at $\phi(\mathfrak{n}')$.

\item For every internal node $\mathfrak{n}$ in $\tau$, and every leaf $f$ which is a child of $\mathfrak{n}$ in $\tau$, we define $t_{\mathfrak{n}\to f}$ to be the tree of $t_{\phi(\mathfrak{n})}$ containing $\phi(f)$.

\end{itemize}

\begin{figure}[htbp]
\begin{center}
\includegraphics[scale=0.6]{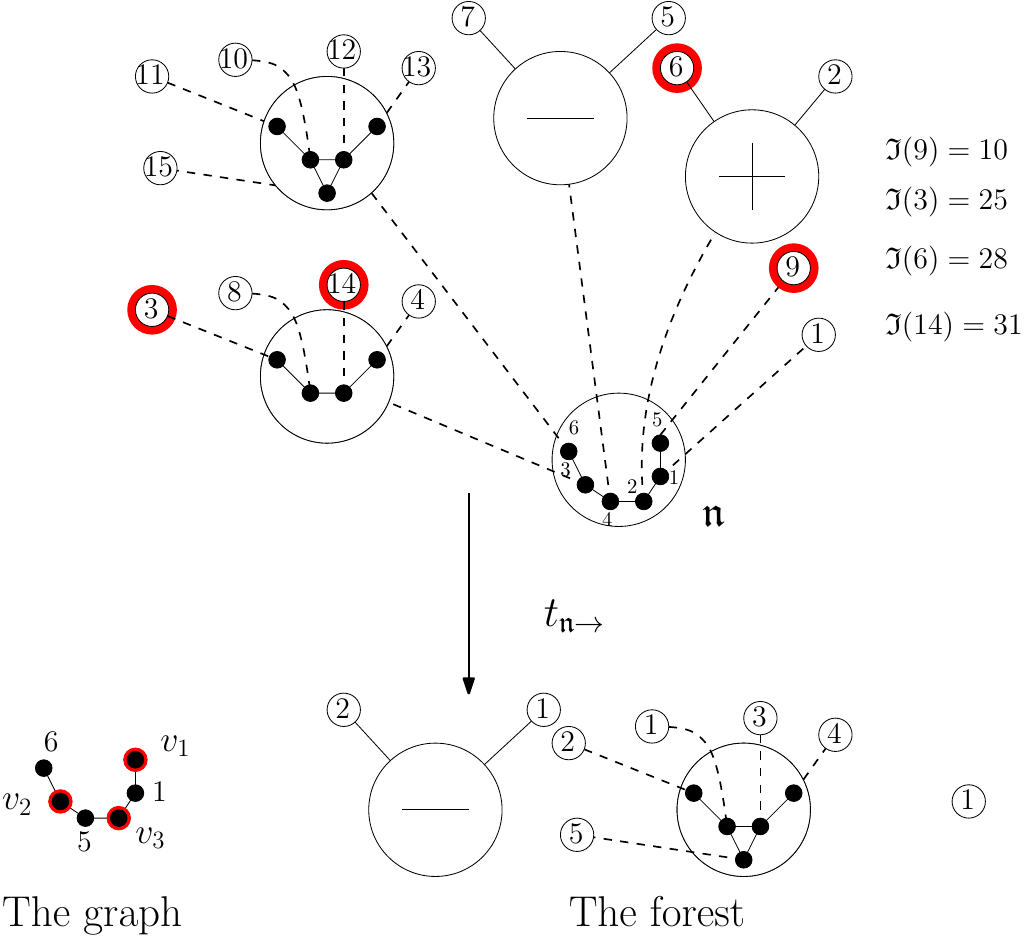}
\caption{Example of the construction of $t_{n\to}$ when $n\in \Vs$ (case (x)).}\label{Cas(x)}
\end{center}
\end{figure}

We now characterize the trees that appear in this decomposition and compute the corresponding exponential generating function. In the rest of the proof, we say abusively that every blossomed tree belongs to $\mathcal{T}_{\mathcal{P}}$, and that two nodes both decorated with $\oplus$ or $\ominus$ have the same decoration, even if they do not have the same number of children.

\noindent{\bf (i): analysis of $t_{\to\mathrm{root}}$}

The tree $t_{\to\mathrm{root}}$ is a tree in $\mathcal{T}_{\mathcal{P}}$, it has no marked leaf and a unique blossom. If the root is not in \Vsr and decorated with $\oplus$ (resp.~$\ominus$), the blossom is $\oplus$-replaceable (see \Cref{dfnblossom}) (resp.~$\ominus$-replaceable). If the root is in \Vsr, the blossom is replaceable.

The corresponding exponential generating function is equal to $T^{\oplus}$ if the root is not in \Vsr and equal to $T^{\mathsf{blo}}$ otherwise.

\noindent{\bf (ii): analysis of $t_{\mathfrak{n}\to \mathfrak{n}'}$ where $\mathfrak{n}\not\in \Vs$ and $\mathfrak{n}'$ is a child of $\mathfrak{n}$ not in \Vsr with the same decoration}

The tree $t_{\mathfrak{n}\to \mathfrak{n}'}$ is a tree in $\mathcal{T}_{\mathcal{P}}$ whose root is not decorated with the same decoration as $\mathfrak{n}$ and with one blossom $\oplus$-replaceable if $\mathfrak{n}'$ is decorated with $\oplus$, $\ominus$-replaceable otherwise and no marked leaf. 

The exponential generating function of such trees is either $T_{\mathrm{not}\oplus}^{\oplus}$ if both nodes are decorated with $\oplus$ or $T_{\mathrm{not}\ominus}^{\ominus}$ ($=T_{\mathrm{not}\oplus}^{\oplus}$) if both nodes are decorated with $\ominus$.

\noindent{\bf (iii): analysis of $t_{\mathfrak{n}\to \mathfrak{n}'}$ where $\mathfrak{n}\not\in \Vs$ and $\mathfrak{n}'$ is a child of $\mathfrak{n}$ not in \Vsr with a different decoration}

The tree $t_{\mathfrak{n}\to \mathfrak{n}'}$ is a tree in $\mathcal{T}_{\mathcal{P}}$ whose root has a different decoration from that of $\mathfrak{n}$ and with one blossom $\oplus$-replaceable if $\mathfrak{n}'$ is decorated with $\oplus$, $\ominus$-replaceable otherwise and no marked leaf.

The exponential generating function of such trees is either $T_{\mathrm{not}\oplus}^{\ominus}$ if $\mathfrak{n}$ is decorated with $\oplus$ and $\mathfrak{n}'$ with $\ominus$ or $T_{\mathrm{not}\ominus}^{\oplus}$ ($=T_{\mathrm{not}\oplus}^{\ominus}$) if $\mathfrak{n}$ is decorated with $\ominus$ and $\mathfrak{n}'$ with $\oplus$.

\noindent{\bf (iv): analysis of $t_{\mathfrak{n}\to \mathfrak{n}'}$ where $\mathfrak{n}\in \Vs$ and $\mathfrak{n}'$ is a child of $\mathfrak{n}$ not in \Vsr }

The tree $t_{\mathfrak{n}\to \mathfrak{n}'}$ is a tree in $\mathcal{T}_{\mathcal{P}}$ with one blossom $\oplus$-replaceable if $v'$ is decorated with $\oplus$, $\ominus$-replaceable otherwise and no marked leaf.

The exponential generating function of such trees is either $T^{\oplus}$ if $\mathfrak{n}'$ is decorated with $\oplus$ or $T^{\ominus}$ ($=T^{\oplus}$) if $\mathfrak{n}'$ is decorated with $\ominus$.

\noindent{\bf (v): analysis of $t_{\mathfrak{n} \to \mathfrak{n}'}$ where $\mathfrak{n}\not\in \Vs$ and $\mathfrak{n}'$ is a child of $\mathfrak{n}$ in \Vsr}

The tree $t_{\mathfrak{n}\to \mathfrak{n}'}$ is a tree in $\mathcal{T}_{\mathcal{P}}$ whose root is not decorated with the decoration of $\mathfrak{n}$ with one blossom and no marked leaf. 

The corresponding exponential generating function is $T_{\mathrm{not}\oplus}^{\mathsf{blo}}$.

\noindent{\bf (vi): analysis of $t_{\mathfrak{n} \to \mathfrak{n}'}$ where $\mathfrak{n}\in \Vs$ and $\mathfrak{n}'$ is a child of $\mathfrak{n}$ in \Vsr}

The tree $t_{\mathfrak{n}\to \mathfrak{n}'}$ is a tree in $\mathcal{T}_{\mathcal{P}}$  with one blossom and no marked leaf. 

The corresponding exponential generating function is $T^{\mathsf{blo}}$.

\noindent{\bf (vii): analysis of $t_{\mathfrak{n}\to f}$ where $\mathfrak{n}\not\in \Vs$ and $f$ is a leaf which is a child of $\mathfrak{n}$}

The tree $t_{\mathfrak{n}\to f}$ is a tree in $\mathcal{T}_{\mathcal{P}}$ whose root is not decorated with the decoration of $\mathfrak{n}$ with one marked leaf and no blossom.

The corresponding exponential generating function is $zT_{\mathrm{not}\oplus}'$.

\noindent{\bf (viii): analysis of $t_{\mathfrak{n}\to f}$ where $\mathfrak{n}\in \Vs$ and $f$ is a leaf which is a child of $\mathfrak{n}$}

The tree $t_{\mathfrak{n}\to f}$ is a tree in $\mathcal{T}_{\mathcal{P}}$  with one marked leaf and no blossom.

The corresponding exponential generating function is $zT'$.

\noindent{\bf (ix): analysis of $t_{\mathfrak{n}\to }$ where $\mathfrak{n}\not\in \Vs$ }

The tree $t_{\mathfrak{n}\to}$ is a tree whose root denoted is decorated with the same decoration as $\mathfrak{n}$, who has no marked leaf and no blossom. It verifies all the conditions of being $(\mathcal{P}$-consistent, except that the root can have $0$ or $1$ child.

The corresponding exponential generating function is $\sum\limits_{k\geq 0}T_{\mathrm{not}\oplus}^k=\exp(T_{\mathrm{not}\oplus})$.

\noindent{\bf (x): analysis of $t_{\mathfrak{n}\to}$ where $\mathfrak{n}\in \Vs$ }
This case is the more subtle one and the main reason we assume $\mathcal{P}$ to be stable by relabelling throughout the paper.

The tree $t_{\mathfrak{n}\to}$ is composed of a graph $G$ with $k$ blossoms, and a forest of $|G|-|\mathrm{dec}(\mathfrak{n})|$ trees.

Let $n$ be an integer such that $n\geq |\mathrm{dec}(\mathfrak{n})|$. We need to compute the cardinality of the set $B_n$ of all graphs $G$ of size $n$ that can be obtained in $t_{n\to}$. To that extent, we define $A_{n,\mathrm{dec}(\mathfrak{n})}$ as the set of marked graphs $(H,\mathfrak{J})$ such that $H$ is in $\mathcal{P}_n$ and the subgraph of $H$ induced by $\mathfrak{J}$ is $\mathrm{dec}(\mathfrak{n})$.

Define the function $\phi: A_{n,\mathrm{dec}(\mathfrak{n})}\to B_n$ such that $\phi(H,\mathfrak{J})$ is the graph with $k$ blossoms obtained after blossoming successively $H$ at $\mathfrak{J}^{-1}(1),\dots, \mathfrak{J}^{-1}(k)$. By construction $\phi$ is well-defined and surjective.

Moreover, for every $G\in B_n$ , since $\mathcal{P}$ is stable by relabelling, the size of $\phi^{-1}(G)$ is the same as the number of choices of labels of the $k$ blossoms: $\frac{n!}{(n-|\mathrm{dec}(\mathfrak{n})|)!}$.

Thus for every $n\geq |\mathrm{dec}(\mathfrak{n})|$, the number of graphs in $B_n$ is 
$$\frac{(n-|\mathrm{dec}(\mathfrak{n})|)!}{n!}| A_{n,\mathrm{dec}(\mathfrak{n})}|=\frac{(n-|\mathrm{dec}(\mathfrak{n})|)!}{n!}\sum\limits_{H'\in\mathcal{P}_n}\mathrm{Occ}_{\mathrm{dec}(\mathfrak{n})}(H').$$

Since the generating function of the forest is $\frac{T^{n-|\mathrm{dec}(\mathfrak{n})|}}{(n-|\mathrm{dec}(\mathfrak{n})|)!}$, the corresponding generating function of $t_{\mathfrak{n}\to}$ is 

$$\sum\limits_{H'\in \mathcal{P}}\frac{\mathrm{Occ}_{\mathrm{dec}(\mathfrak{n})}(H')T^{|H'|-|\mathrm{dec}(\mathfrak{n})|}}{|H'|!}=\mathrm{Occ}_{\mathrm{dec}(\mathfrak{n}), \mathcal{P}}(T)$$\medskip

Note that, with all these conditions, the process described in this proof is a bijection between $\mathcal{T}_{\tau, \Vs}$ and the labelled forest of all previous mentioned trees (without doing reduction). Here is a reciprocal to recover the original tree.

Simultaneously, 
\begin{itemize}
\item For every internal node $v$ of $\tau$ in $\Vs$, we replace $t_{v\to}$ by a tree whose root is decorated with the graph of $t_{v\to}$ denoted by $G$, and whose subtrees are composed of the forest in $t_{v\to}$ and all the different trees of the form $t_{\mathfrak{n}, \mathfrak{n}_0}$ and $t_{\mathfrak{n}\to \ell}$. We change the labels of the blossoms of the decoration such that for every $j$ such that $*_j$ is a blossom of the decoration, the label $*_j$ is replaced by the rank of $t_{\mathfrak{n}\to \mathfrak{n}_j}$ in the set of subtrees, where $\mathfrak{n}_j$ is the $j$-th children of $\mathfrak{n}$ in $\tau$. The other labels are replaced so that the decoration is labeled, and the order between the previous and the new label stay the same;

    \item For every internal node $\mathfrak{n}$ of $\tau$ not in $\Vs$, we add to the subtrees of the root of $t_{\mathfrak{n}\to }$ the different trees of the form $t_{\mathfrak{n}, \mathfrak{n}_0}$ and $t_{\mathfrak{n}\to \ell}$. 
    \item For every internal node, we replace the blossom of $t_{\mathfrak{n}'\to \mathfrak{n}}$ with $\mathfrak{n}'$ the parent of $v$ (resp. $t_{\to \mathrm{root}}$ if $\mathfrak{n}$ is the root of $\tau$) by the tree $t_{\mathfrak{n}\to}$ after the previous modifications to get a tree $t'$. For every internal decoration of a node $\mathfrak{n}_0$ blossomed at some $t_{\mathfrak{n}'\to \mathfrak{n}}$, we replace the decoration of the blossom with the rank $r$ of the tree of $t'_{\mathfrak{n}_0}$ containing $t_{\mathfrak{n}\to}$, and increase all the labels greater or equal to $r$ by one.

\end{itemize}

After these operations, we recover the original tree.

\begin{figure}
\begin{center}
\centering
\includegraphics[scale=0.6]{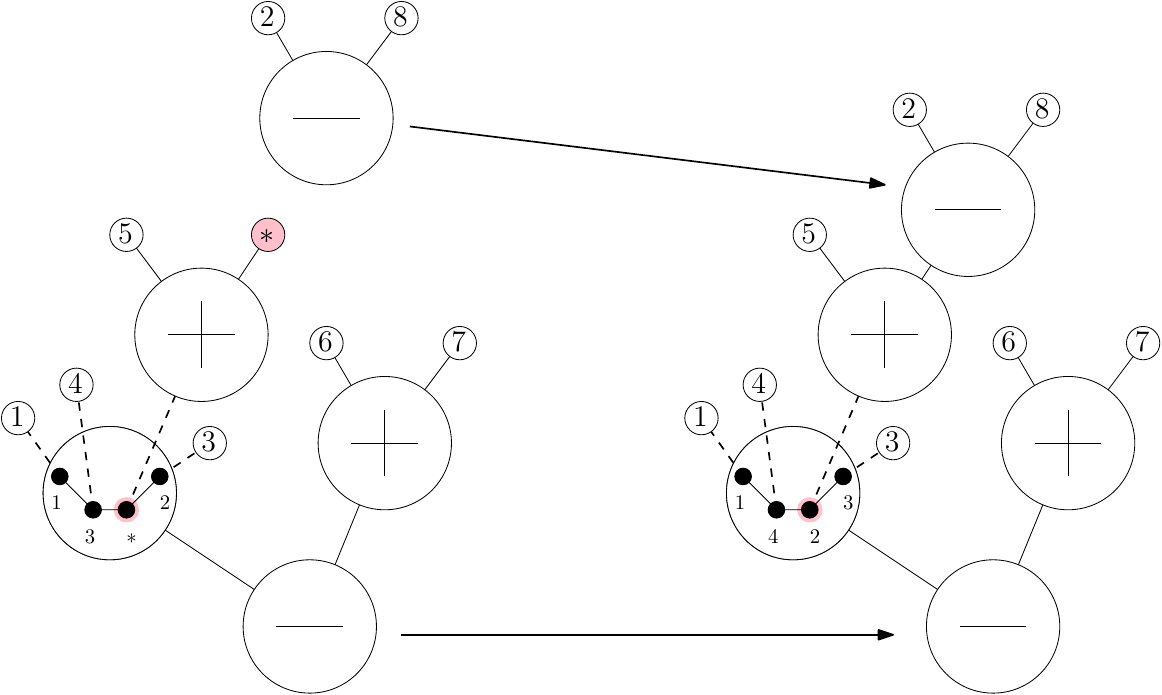}
\caption{Gluing of two trees.}
\end{center}
\end{figure}

Thus, $T_{\tau,\Vs}$ is the product of the generating functions and this concludes the proof of the theorem.

\end{proof}

\begin{cor}\label{cort}
    With the notations of \Cref{thmcoeur}, 
    \begin{align*}
T_{\tau,\Vs}=z^{|\tau|}&T^{\mathrm{root}}\left(T^{\oplus}\right)^{e}\exp((d_{\Vs\to \ell}+d_{\Vs \to \Vs}+n_L-d_{=}-d_{\neq}) T_{\mathrm{not}\oplus})\prod\limits_{\mathfrak{n}\in \Vs}\mathrm{Occ}_{\mathrm{dec}(\mathfrak{n}),\mathcal{P}}(T)
\end{align*}
where $e$ is the number of edges of $\tau$.
\end{cor}

\begin{proof}
Simple consequence of \cref{ecu123,ecu124,b1,b2,b3,o1,o2} and \Cref{thmcoeur}.
\end{proof}

\begin{cor}\label{corrr}
Under Condition (C) defined p.\pageref{condi}, the series $T_{\tau,\Vs}$ has radius of convergence $R$, is $\Delta$-analytic and its asymptotic expansion near $R$ is: 
$$T_{\tau,\Vs}=\frac{(1+K)}{\mu}C_{\tau,\Vs}\left(1-\frac{z}{R}\right)^{-\frac{e+1}{2}}\left(1+o(1)\right)$$
where
$$C_{\tau,\Vs}:= \frac{(1+K)^{f}R^{|\tau|}}{\mu^{e}}\times \prod\limits_{\mathfrak{n}\in \Vs}\mathrm{Occ}_{\mathrm{dec}(\mathfrak{n}),\mathcal{P}}(K)$$
where $e$ is the number of edges of $\tau$, $f$ the number of edges between two nodes $(n,n')$ such that $n'$ is a child of $n$ and $n$ is in $\Vs$ and $\mu$ and $K$ are defined by \cref{mu,K}.

\end{cor}

\begin{proof}
Note that by counting the edges $(\mathfrak{n},\mathfrak{n}')$, where $\mathfrak{n}'$ is a child of $\mathfrak{n}$ and $\mathfrak{n}'$ is not in \Vsr $n_L=d_{=}+d_{\neq}+d_{\Vs \to \overline{\Vs}}$ if the root is in \Vsr, $n_L=d_{=}+d_{\neq}+d_{\Vs \to \overline{\Vs}}+1$ otherwise. Using this, the corollary is a simple consequence of \cref{cort} and \cref{toplus}.
\end{proof}

Summing over all choices of \Vsr gives the following corollary

\begin{cor}\label{corrrr}
Under Condition (C) defined p.\pageref{condi}, the series $T_{\tau}$ has radius of convergence $R$, is $\Delta$-analytic and its asymptotic expansion near $R$ is: 
$$T_{\tau}=\frac{(1+K)}{\mu}B_{\tau}\left(1-\frac{z}{R}\right)^{-\frac{e+1}{2}}\left(1+o(1)\right)$$
where
$$B_{\tau}:= \frac{R^{|\tau|}}{\mu^{e}}\times \prod\limits_{\mathfrak{n}\in \Vs_{\mathrm{nl}}}\mathrm{Occ}_{\mathrm{dec}(\mathfrak{n}),\mathcal{P}}(K)\prod\limits_{\mathfrak{n}\in \Vs_{\mathrm{l}}}\left(\mathrm{Occ}_{\mathrm{dec}(\mathfrak{n}),\mathcal{P}}(K)(1+K)^{d_{\mathfrak{n}}}+1\right)$$

where $e$ is the number of edges of $\tau$, $d_{\mathfrak{n}}$ the number of children of $\mathfrak{n}$ for every internal node $v$, $\Vs_{\mathrm{nl}}$ is the set of non-linear nodes of $\tau$ and $\Vs_{\mathrm{l}}$ is the set of linear nodes of $\tau$.

\end{cor}

\section{Main results}\label{lastsection}

\begin{thm}\label{110}

Let $\tau$ be a tree with $\ell\geq 2$ leaves. For $n\geq \ell$ and $\mathbf{T}^{(n)}$ be a uniform random tree in $\mathcal{T}_{\mathcal{P}}$ with $n$ vertices.
Let $\mathbf{\mathfrak{I}_{\ell}}^{(n)}$ be a uniform partial injection from $\{1,\dots,n\}$ to $\N$ whose image is $\{1,\dots, \ell\}$ and independent of $\mathbf{T}^{(n)}$. Denote by $\mathbf{T}_{\mathbf{\mathfrak{I}_{\ell}}^{(n)}}^{(n)}$ the subtree induced by $\mathbf{\mathfrak{I}_{\ell}}^{(n)}$. Assume that Condition (C) defined p.\pageref{condi} holds.

If $\tau$ is a binary tree then $$\mathbb{P}(\mathbf{T}_{\mathbf{\mathfrak{I}_{\ell}}^{(n)}}^{(n)}=\tau) \xrightarrow[n\to\infty]{} \frac{(\ell-1)!}{(2(\ell-1))!}2^{\ell-1}p^{\Vs_+}(1-p)^{\Vs_-}$$ 
where $\Vs_+$ (resp.~$\Vs_-$) is the number of internal nodes of $\tau$ decorated with $\oplus$ (resp.~$\ominus$) and
\begin{align}\label{parameterp}
   p&=\frac{1+(1+K)^2\mathrm{Occ}_{\oplus_2,\mathcal{P}}(K)}{\Lambda''(\kappa)} 
\end{align}
\noindent where $\oplus_2$ (resp.~$\ominus_2$) is the graph with $2$ vertices and $1$ (resp.~$0$) edge.

If on the contrary $\tau$ is not binary then $$\mathbb{P}(\mathbf{T}_{\mathbf{\mathfrak{I}_{\ell}}^{(n)}}^{(n)}=\tau) \xrightarrow[n\to\infty]{} 0$$

\end{thm}

\begin{proof}
Since $\mathbf{\mathfrak{I}_{\ell}}^{(n)}$ is independent of $\mathbf{T}^{(n)}$, $$\mathbb{P}(\mathbf{T}_{\mathbf{\mathfrak{I}_{\ell}}^{(n)}}^{(n)}=\tau)=\frac{n![z^n]T_{\tau}}{n(n-1)\dots (n-\ell+1)n![z^n]T}=\frac{[z^n]T_{\tau}}{n(n-1)\dots (n-\ell+1)[z^n]T}$$

By applying the Transfer Theorem \cite[Corollary VI.1 p.392]{flajolet2009analytic} to \cref{corrr}, we get 
$$[z^n]T_{\tau}\sim \frac{(1+K)}{\mu} B_{\tau} \frac{n^{\frac{e-1}{2}}}{\Gamma\left(\frac{e+1}{2}\right)R^n}$$

and by \cref{cor1} we obtain
$$n\times \cdots\times (n-\ell+1)[z^n]T\sim n^{\ell} \frac{(1+K)R}{\sqrt{\pi}\mu} \frac{1}{R^{n}n^{\frac{3}{2}}}.$$

Thus when $n$ goes to infinity 
$$\mathbb{P}(\mathbf{T}_{\mathbf{\mathfrak{I}_{\ell}}^{(n)}}^{(n)}=\tau) \sim \frac{B_{\tau}\sqrt{\pi}}{R\Gamma\left(\frac{e+1}{2}\right)}n^{\frac{e+2}{2}-\ell}.$$

Hence if $\tau$ is not a binary tree, $\mathbb{P}(\mathbf{T}_{\mathbf{\mathfrak{I}_{\ell}}^{(n)}}^{(n)}=\tau) \xrightarrow[n\to\infty]{} 0.$

Assume that $\tau$ is a binary tree. Since here $e=2\ell-2$, we get when $n$ goes to infinity \[\mathbb{P}(\mathbf{T}_{\mathbf{\mathfrak{I}_{\ell}}^{(n)}}^{(n)}=\tau)\to \frac{\sqrt{\pi}B_{\tau}}{R\Gamma\left(\frac{2\ell-1}{2}\right)}=\frac{(\ell-1)!}{(2(\ell-1))!}\frac{2^{2\ell-2}B_{\tau}}{R}\] 

Since a binary tree with $\ell$ leaves has $\ell-1$ internal nodes, we have:
\begin{align*}\frac{2^{2\ell-2}B_{\tau}}{R}&\frac{2^{2\ell-2}R^{\ell-1}}{(2R\Lambda''(\kappa))^{\ell-1}}\prod\limits_{\mathfrak{n}\in \Vs_{\mathrm{l}}}\left(\mathrm{Occ}_{\mathrm{dec}(\mathfrak{n}),\mathcal{P}}(K)(1+K)^{d_{\mathfrak{n}}}+1\right)\\
&=\frac{2^{\ell-1}}{(\Lambda''(\kappa))^{\ell-1}}\left(1+(1+K)^2\mathrm{Occ}_{\oplus_2,\mathcal{P}}(K)\right)^{\Vs_+}\left(1+(1+K)^2\mathrm{Occ}_{\ominus_2,\mathcal{P}}(K)\right)^{\Vs_-}\\
&=2^{\ell-1}\left(\frac{1+(1+K)^2\mathrm{Occ}_{\oplus_2,\mathcal{P}}(K)}{\Lambda''(\kappa)}\right)^{\Vs_+}\left(\frac{1+(1+K)^2\mathrm{Occ}_{\ominus_2,\mathcal{P}}(K)}{\Lambda''(\kappa)}\right)^{\Vs_-}\\
&=2^{\ell-1}p^{\Vs_+}(1-p)^{\Vs_-}.
\end{align*}

Last equality holds because

\begin{align*}\frac{1+(1+K)^2\mathrm{Occ}_{\oplus_2,\mathcal{P}}(K)}{\Lambda''(\kappa)}+\frac{1+(1+K)^2\mathrm{Occ}_{\ominus_2,\mathcal{P}}(K)}{\Lambda''(\kappa)}&=\frac{2+(1+K)^2\sum\limits_{G\in\mathcal{P}}\frac{|G|(|G|-1)z^{|G|-2}}{|G|!}}{\Lambda''(\kappa)}\\
&=\frac{2+(1+K)^2P''(K)}{\Lambda''(\kappa)}=1
\end{align*}

since
\begin{align*}\Lambda''(\kappa)&=\exp(\kappa)^2P''(\exp(\kappa)-1)+\exp(\kappa)+P'(\exp(\kappa)-1)+\exp(\kappa)\\
&=(1+K)^2P''(K)+1+\Lambda'(\kappa)=2+(1+K)^2P''(K).
\end{align*}

Combining all the previous equality gives the announced result.

\end{proof}

\Cref{cgbintro} is a corollary of \cref{110} and \cref{111}.

\begin{thm}\label{thmprime}
 Let $H$ be a graph. For $n\geq |H|$, let $\mathbf{G}^{(n)}$ be a uniform random graph in $\mathcal{G}_{\mathcal{P}}$ with $n$ vertices. Let $(d_1,\dots,d_k)$ be the degrees of the linear nodes of the modular decomposition tree of $H$.

Then under Condition (C) defined p.\pageref{condi} 

$$\mathbb{E}[\mathrm{Occ}_{H}(\mathbf{G}^{(n)})]\sim K_H n^{|H|-\beta(H)}$$
where $\beta(H)$ is defined in \cref{betaG} and with $$K_H=\frac{\sqrt{\pi}\prod\limits_{i=1}^k(2d_i-3)!!}{2^{|H|-1-\beta(H)}\Gamma\left(\frac{2|H|-1-2\beta(H)}{2}\right)}\left(\prod\limits_{\mathfrak{n}\in \Vs_{\mathrm{nl}}}\frac{\mathrm{Occ}_{\mathrm{dec}(\mathfrak{n}),\mathcal{P}}(K)R^{|\mathrm{dec}(\mathfrak{n})|-2}}{\Lambda''(\kappa)^{\mathrm{dec}(\mathfrak{n})/2}}\right)p^{d_{\oplus}-n_{\oplus}} (1-p)^{d_{\ominus}-n_{\ominus}}$$

\noindent where $d_1,\dots,d_k$ are the number of children of every linear node in the modular decomposition tree of $H$, $n_{\oplus}$ (resp.~$n_{\ominus}$) is the number of internal node decorated with $\oplus$ (resp.~$\ominus$) in the modular decomposition tree of $H$, and $d_{\oplus}$ (resp.~$d_{\ominus}$) is the sum of the number of children of internal nodes decorated with $\oplus$ (resp.~$\ominus$) in the modular decomposition tree of $H$.
\end{thm}

\begin{proof}
    Let $\mathcal{T}_H$ be the set of substitution tree $t$ such that $\mathrm{Graph}(t)=H$. Note that $$\mathbb{E}[\mathrm{Occ}_{H}(\mathbf{G}^{(n)})]=\frac{n![z^n]\sum\limits_{\tau\in\mathcal{T}_H}T_{\tau}}{n![z^n]T}$$

    Following the proof of \cref{110}, for every $\tau\in\mathcal{T}_H$, 
    $$\frac{[z^n]T_{\tau}}{[z^n]T} \sim \frac{B_{\tau}\sqrt{\pi}}{R\ \Gamma\left(\frac{e+1}{2}\right)}n^{\frac{e+2}{2}}.$$

    Using the previous asymptotic, it yields that in the expansion of $\mathbb{E}[\mathrm{Occ}_{H}(\mathbf{G}^{(n)})]$, only the terms corresponding to an expanded tree in $\mathcal{T}_H$. Thus by \cref{exp}
\begin{align*}
   \mathbb{E}[\mathrm{Occ}_{H}(\mathbf{G}^{(n)})]&\sim  \frac{R^{|H|}\sqrt{\pi}\prod\limits_{i=1}^k(2d_i-3)!!}{R\mu^{2|H|-2-2\beta(H)}\Gamma\left(\frac{2|H|-1-2\beta(H)}{2}\right)}\prod\limits_{\mathfrak{n}\in \Vs_{\mathrm{nl}}}\mathrm{Occ}_{\mathrm{dec}(\mathfrak{n}),\mathcal{P}}(K) \\
   &\quad \times \prod\limits_{\mathfrak{n}\in \Vs_{\mathrm{l}}}\left(\mathrm{Occ}_{\mathrm{dec}(\mathfrak{n}),\mathcal{P}}(K)(1+K)^{2}+1\right)n^{|H|-\beta(H)}
\end{align*}

    Using the explicit expression of $\mu$, and the fact that in an expanded tree, every linear node of size $k$ is replaced with a binary tree with $k-1$ binary node with the same decoration, we get 
    $$\mathbb{E}[\mathrm{Occ}_{H}(\mathbf{G}^{(n)})]\sim  K_Hn^{|H|-\beta(H)}$$
    with $$K_H=\frac{\sqrt{\pi}\prod\limits_{i=1}^k(2d_i-3)!!}{2^{|H|-1-\beta(H)}\Gamma\left(\frac{2|H|-1-2\beta(H)}{2}\right)}\left(\prod\limits_{\mathfrak{n}\in \Vs_{\mathrm{nl}}}\frac{\mathrm{Occ}_{\mathrm{dec}(\mathfrak{n}),\mathcal{P}}(K)R^{|\mathrm{dec}(\mathfrak{n})|-2}}{\Lambda''(\kappa)^{\mathrm{dec}(\mathfrak{n})/2}}\right)p^{d_{\oplus}-n_{\oplus}} (1-p)^{d_{\ominus}-n_{\ominus}}.$$
\end{proof}

\textbf{Acknowledgements.} I would like to thank Lucas Gerin and Frédérique Bassino for useful discussions and for carefully reading many earlier versions of this manuscript.

\bibliographystyle{plain}
\bibliography{biblio}

\vfill 
\noindent \textsc{Théo Lenoir} \verb|theo.lenoir@polytechnique.edu|\\
\textsc{Cmap, Cnrs}, \'Ecole polytechnique,\\
Institut Polytechnique de Paris,\\
91120 Palaiseau, France

\end{document}